\documentclass[11pt]{amsart}
\usepackage{amssymb}
\usepackage{amsmath} 
\usepackage{amsthm} 
\usepackage{epsfig}
\usepackage{graphicx}
\usepackage{color}
\usepackage{transparent}
\usepackage{soul}
\usepackage{subfig}
\usepackage{wrapfig}
\usepackage[colorinlistoftodos]{todonotes}
\usepackage[colorlinks=true,linkcolor=blue,citecolor=blue,urlcolor=blue]{hyperref}
\usepackage[title]{appendix}
\usepackage{listings}
  \usepackage{mathrsfs}
 \usepackage{tikz}
 \usetikzlibrary{arrows.meta,positioning,calc}

\numberwithin{equation}{section}

\newcommand{\vs}{\mathbf{s}}

\newcommand{\conW}{\mathrm{Conj}(W(E_{10}))}

\newtheorem*{thmm}{Theorem}

\newtheorem*{conj}{Conjecture}

\newtheorem{thm}{Theorem}[section]
\newtheorem{lem}[thm]{Lemma}
\newtheorem{defn}[thm]{Definition}
\newtheorem{cor}[thm]{Corollary}
\newtheorem{rem}[thm]{Remark}

\newtheorem{prop}[thm]{Proposition}

\newcommand{\Max}{\mathrm{Max}}
\newcommand{\Min}{\mathrm{Min}}
\newcommand{\vve}{\mathbf{e}}

\title{Spectral Growth in $W(E_{10})$: \\ Double Coset Filtration and Hilbert Geometry} 

\author{Kyounghee Kim}
\address{Department of Mathematics\\
         Florida State University\\
         Tallahassee, FL 32308}
\email{kkim6@fsu.edu}


\begin{document}
\maketitle

\begin{abstract}
We study the spectral radii of elements in the hyperbolic Coxeter group $W(E_{10})$ by introducing a filtration indexed by reflections conjugate to a distinguished simple reflection $s_0$. This filtration organizes $W(E_{10})$ into double cosets relative to the parabolic subgroup $W(A_9)$, and we classify the minimal representatives of these cosets via a rooted directed acyclic graph (DAG) labeled by triples. Each node in the DAG corresponds to a structured reflection composition, enabling a recursive understanding of spectral growth. Using the Hilbert metric on the Tits cone, we relate spectral radii to geometric displacement and demonstrate an effective method to compute the spectral radii inductively. This provides a geometric and combinatorial framework for understanding the Weyl spectrum of $W(E_{10})$. While our focus is on $E_{10}$, the techniques developed extended naturally to the family $W(E_n)$ for $n\ge 10$, with implications for dynamics on rational surfaces and entropy spectra of surface automorphisms.
\end{abstract}

\section{Introduction}\label{S:intro}
Hyperbolic Coxeter groups form a cornerstone of modern geometric group theory, with deep connections to hyperbolic geometry, Kac-Moody algebras, and arithmetic reflection groups. Since the foundational work of Vinberg~\cite{Vinberg:1971} on hyperbolic reflection groups, these groups have served as central examples in the study of discrete symmetries and tessellations of non-Euclidean spaces. Among them, the Coxeter group $W(E_{10})$ stands out as the unique simply-laced hyperbolic Coxeter group of rank~$10$. It plays a prominent role in both geometric and arithmetic settings, and appears naturally in the spectral dynamics of automorphisms of rational surfaces of Picard rank~$11$~\cite{McMullen:2002,McMullen:2007,Bedford-Kim:2009,dolgachev2008reflection,Kim:2023}.

\medskip
This article introduces a new framework for organizing and analyzing the spectral behavior of elements in $W(E_{10})$, based on a filtration indexed by the distinguished simple reflection $s_0$ at the branching node of the $E_{10}$ Dynkin diagram. Our approach centers on the structure of conjugates of $s_0$, placing it in close proximity to the theory of reflection length. The reflection length in infinite and hyperbolic Coxeter groups has received substantial attention in recent years; see, for example,~\cite{McCammond-Petersen:2011,Duszenko:2012,Drake-Peters:2021,Lotz:2025}.

%
%
%

\medskip
We define the \emph{$s_0$--level} of an element $\omega \in W(E_{10})$, denoted $h_{s_0}(\omega)$, as the minimal number of occurrences of $s_0$ in any reduced word among all conjugates of $\omega$. This statistic induces a natural \emph{filtration by double cosets}:
\[ \mathcal{D}_{n,n} := \{ \omega \in W(E_{10})\ :\  \omega \in W(A_9) \kappa W(A_9)\ \ \text{with } h_{s_0}(\omega) =n\},\] where $\kappa$ is the unique Bruhat-minimal representative of its double coset.

In Section~\ref{S:doubleCosets} and Section~\ref{S:Ktree}, we give a combinatorial classification of these minimal representatives using ordered triples from the set \[T =\{ {i,j,k} \subset \{1,\dots, 10\} \ : 1 \le i<j<k \le 10\}. \] For each $s_0$--level $n$, the corresponding minimal representatives are encoded by an inductively defined finite set $T_n \subset T^n$ of ordered $n$--tuples of triples. This leads to a natural identification of minimal double coset representatives with nodes in a \emph{rooted directed acyclic graph} (DAG), where each node is labeled by a triple in $T$, and directed edges represent the addition of a new triple during the recursive construction.

\begin{thmm}[Triple-Labeled Graph Structure of Minimal Representatives, Theorem~\ref{T:TripleGraph}]
Let $\mathcal{D}_{n,n}$ denote the set of elements in $W(E_{10})$ of $s_0$--level $n$ (as defined in~\eqref{E:Dnn}), and let $K \subset W(E_{10})$ be the set of Bruhat-minimal representatives of double cosets \(W(A_9)\backslash W(E_{10})/W(A_9)\).

Then:
\begin{enumerate}
\item Each $\kappa \in K \cap \mathcal{D}_{n,n}$ is uniquely determined by an ordere $n$--tuple of triples from
 \[T\ :=\ \{\{i,j,k\} \subset \{1,\dots, 10\} \ :\ 1\le i<j<k \le 10\}.\]
\item There exists a rooted directed acyclic graph (DAG) $\mathcal{G}$, whose depth-$n$ nodes correspond bijectively to a finite inductively defined subset $T_n \subset T^n$, such that each directed path of length $n$ encodes the recursive construction of a minimal double coset representative $\kappa \in \mathcal{D}_{n,n}.$
\end{enumerate}
\end{thmm}

\medskip
This DAG encodes the combinatorial growth of minimal representatives and provides a canonical framework for tracking how new reflections from $s_0$--conjugacy classes are introduced. Each directed path in the graph corresponds to a specific representative $\kappa \in W(E_{10})$ with $s_0$--level $n$, and thus serves as a backbone for understanding the spectral stratification of $W(E_{10})$.

%
%

\medskip
Associated to each level $n$, we define a finite \emph{spectral set}:
\[ \Lambda_n\ :=\ \{ \rho(\omega) \ :\  \omega \in \mathcal{D}_{n,n}\},\] consisting of the spectral radii of elements constructed from $n$ reflections in $s_0$--conjugacy classes. This construction yields a new filtration iof the spectral data of $W(E_{10})$, revealing levelwise structure and exponential growth. 
\medskip

\begin{thmm}[Spectral Monotonicity, Theorem~\ref{T:monotonicity}]
Let $\{\Lambda_k(\widehat I^{(k)})\}_{k\ge1}$ be the sequence of
spectral subsets associated with a directed path in the
$s_0$--growth graph. Then the sequences of minima and maxima are non-decreasing in $s_0$--level $k$:
\[
\min \Lambda_k(\widehat I^{(k)})\;\le\;
\min \Lambda_{k+1}(\widehat I^{(k+1)}), \qquad
\max \Lambda_k(\widehat I^{(k)})\;\le\;
\max \Lambda_{k+1}(\widehat I^{(k+1)})
\quad \text{for all } k \ge 1.
\]
\end{thmm}

\begin{thmm}[Spectral Exhaustion Theorem, Theorem~\ref{T:exhaustion}] For every $C >1$, there eixsts an integer $N\ge 1$ such that
\[ \{ \rho(\omega) \in \mathbb{R}\ :\ \omega \in W(E_{10}), \ \rho(\omega)\le C\} \subset \bigcup_{n\le N} \Lambda_n. \]
\end{thmm}

We further interpret this filtration geometrically using the Hilbert metric on the Tits cone of $W(E_{10})$, and prove that the spectral envelopes $\Lambda_n$  grow monotonically in $n$. This leads to a level-wise recursive method for computing all spectral radii in $W(E_{10})$ below a given threshold, with special attention to the emergence and density of primitive Salem numbers.

\begin{thmm}[Hilbert Ball Inclusion of Level-$n$ Elements, Theorem~\ref{T:hilbertBall}]
Let $x$ be a point in the interior of the fundamental chamber.
Define $B_n := B_K(x, t_n) \subset \mathcal{C}$ to denote the Hilbert ball of radius.
\[
t_n := \inf\bigl\{\delta(\omega) \;:\; \omega \in \mathcal{D}_{n,n}\bigr\}.
\]
Then for any $\omega \in W(E_{10})$, if $\omega \cdot x \in B_n$,
the $s_0$--level of $\omega$ satisfies $h_{s_0}(\omega) \le n$.
\end{thmm}

\medskip
Our computations reveal the set of Salem numbers that are contained in $\Lambda_n$. We isolated the primitive part $\hat \Lambda_n \subset \Lambda_n$--- the set of Salem numbers whose $k^\text{th}$ root is not a Salem number for $k\ge 2$---which emerges at level $n$, and observe experimentally that their number grows approximately fivefold at each step: 
\[ |\hat \Lambda_n| \approx 5 \cdot |\hat \Lambda_{n-1}|, \qquad \text{ for } n\ge 2.\]

This spectral growth has a precise geometric manifestation. Via McMullen's theory of Hilbert metrics on the Tits cone $\mathcal{C}$, we prove that the level sets $\mathcal{D}_{n,n}$ are confined outside a sequence of \emph{nested Hilbert balls} with a fixed base point $x$ in the fundamental domain 
\[ B_n \ :=\ B_K(x, t_n), \quad t_n \sim \log m_n,\]
where $m_n := \min \Lambda_n$. Thus, the $s_0$--filtration yields a geometric stratification of the Tits cone that reflects the spectral radius of elements and the depth of their reflection complexity.

The computations and structures developed in this work support the emergence of a rich and previously uncharted spectral hierarchy inside the Coxeter group $W(E_{10})$. Guided by the recursive growth of minimal double coset representatives and their interaction with the Hilbert metric on the Tits cone, our data suggest the following:

\begin{conj}[Spectral Growth Conjecture]
Let $\Lambda_n\subset \mathbb{R}_{>1}$ denote the set of spectral radii of $W(E_{10})$-elements arising at $s_0$--level $n$. Then:
\begin{itemize}
\item[(i)] 
The maximal values satisfy the exact relation \[ M_n \;=\;M_1^n, \quad n\ge 1,\]
as proved in Lemma~\ref{L:maxM}. Moreover, the minimal extreme values appear to obey an approximate geometric scaling:
\[  m_n \;\approx\; \delta^{n-2} M_{n-1}, \qquad \text{for }n\ge3,\]
for some $0<\delta<1$. Here ``$\approx$" indicates numerical proximity.
\item[(ii)] The number of new primitive Salem numbers at level $n$, denoted $|\hat \Lambda_n|$, grows exponentially with $n$:
\[ |\hat \Lambda_n|\; \sim \; C \cdot r^n, \qquad \text{with } r \approx 5.\]
\item[(iii)] The maximal spectral gap between consecutive elements in $\hat \Lambda_n$ decays exponentially, and the primitive spectrum becomes increasingly dense as $n\to \infty$.
\end{itemize}
\end{conj}

These observations suggest that the $s_0$--growth graph encodes not just the combinatorial unfolding of $W(E_{10})$, but also a kind of spectral stratification geometry: new eigenvalues emerge along specific paths, bounded within nested Hilbert balls, and propagate upward in a controlled yet rapidly diversifying manner. We anticipate that this framework will offer new insights into both the representation theory and geometric group theory of hyperbolic Coxeter groups.

\begin{rem}
Kim~\cite{Kim-Salem} investigated the closure (the Weyl spectrum) of the set of spectral radii in $\bigcup_n W(E_n)$, establishing monotonicity with respect to the Coxeter rank $n$. In contrast, the present article fixes the Coxeter rank $n=10$ and studies spectral growth via the $s_0$--level filtration---the minimal number of occurrences of the distinguished reflection $s_0$ in the conjugacy class. 
\end{rem}

\medskip
\noindent\textbf{Connection to Entropy Spectrum of Rational Surface Automorphisms}
Suppose $X$ is a rational surface. By Nagata \cite{Nagata, Nagata2}, an automorphism $f: X\to X$ is a lift of a birational map on $\textbf{P}^2(\mathbb{C})$. Also, by the identification between the intersection product and a symmetric Bilinear form on a Minkowski space, if $f$ has an infinite order, then the induced action $f_*$ on the Picard group of $X$ is an element of the Coxeter group $W(E_n)$ where $n+1$ is the Picard rank of $X$. The topological entropy of $f$ is determined by the spectral radius of $f_*$ as an action on a real vector space:
\[ h_{top}(f) \ =\  \log \rho(f_*|_{{\text{Pic}(X)}}).\]

\medskip
It is known that the logarithm of every spectral radius of $\bigcup_n W(E_n)$ arises as the topological entropy of a rational surface automorphism \cite{Diller:2011, McMullen:2007, Uehara:2010, Kim:2023}. Moreover, the closure of this set--- often called the \emph{Weyl spectrum}--- coincides with the closure of the set of exponentials of topological entropies of rational surface automorphisms \cite{Diller-Favre:2001, blancdynamical}. However, since many elements $\omega \in \bigcup_n W(E_n)$ can share the same spectral radius, identifying which element corresponds to a geometric realization is often nontrivial. 

\medskip
Recent results by Uehara \cite{Uehara:2010} and Kim \cite[Theorem~B]{Kim:2023} show that for each $\omega \in W(E_n)$, there exists a representative $\omega' \in W(E_n)$ obtained by multiplying $\omega$ with elements of the parabolic subgroup $W(A_n) \subset W(E_n)$, such that $\rho(\omega') = \rho(\omega)$ and $\omega'$ is realized as the induced action on the Picard group of a rational surface automorphism $f:X \to X$. In particular, one obtains the identity
\[ S_{10}:=\{h_{top}(f) : f \in \text{Aut}(X),\ \  \text{rank of }\text{Pic}(X) = 11\} = \log \{\rho(\omega) : \omega \in W(E_{10})\}.\]

Our results provide a complete, recursive framework for computing this spectrum. In particular, we determine (recursively) the full set of topological entropies arising from automorphisms of rational surfaces with Picard rank $11$. It is noteworthy that, based on Table~\ref{T:Lambda_1} and Table~\ref{T:Senvelops}, all elements of  $S_{10} \cap (0,0.36) $ appear to arise from quadratic rational surface automorphisms. 

\medskip
\noindent\textbf{Organization.} This article is organized as follows. In Section~\ref{S:coxeter}, we provide a brief review of Coxeter groups and related background relevant to the remainder of the paper. Section~\ref{S:doubleCosets} introduces the notion of $s_0$--length and constructs the Bruhat-minimal representatives for double cosets with respect to the parabolic subgroup $W(A_9)$. In Section~\ref{S:Ktree}, we define the $s_0$--level and describe the associated rooted directed acyclic graph (DAG) that encodes the recursive structure of minimal representatives. Section~\ref{S:Hilbert} presents the main theorems and geometric interpretation using the Hilbert metric, along with supporting numerical results. Appendix~\ref{B:tables} contains a complete list of primitive Salem numbers and their minimal polynomials in $\hat \Lambda_2$, and the first 50 such numbers in $\hat \Lambda_3$. Finally, Appendix~\ref{A:code} includes a ready-to-use SageMath script for generating the ordered triples associated with Bruhat-minimal representatives.

\section{Coxeter group $W(E_{10})$}\label{S:coxeter}

In this section, we briefly summarize the structure and properties of the Coxeter group $W(E_{10})$. While many of these features hold for general Coxeter groups, we focus on those most relevant to our context. For background and comprehensive treatment, see~\cite{BjornerBrenti,Humphreys1990}.

The group $W(E_{10})$ is the simply-laced hyperbolic Coxeter group of type $E_{10}$--a rank-10 indefinite extension of the finite type $E_8$. It is generated by ten involutions $\{s_i : i = 0, 1, \dots, 9\}$, one for each node of the Dynkin diagram in Figure~\ref{F:coxetergraph}, subject to the standard Coxeter relations:
\[
(s_i s_j)^{m_{ij}} = 1,
\]
where $m_{ii} = 1$, $m_{ij} = 3$ if nodes are joined by an edge, and $m_{ij} = 2$ otherwise.

\begin{figure}[h]
\centering
\def\svgwidth{3.1truein}
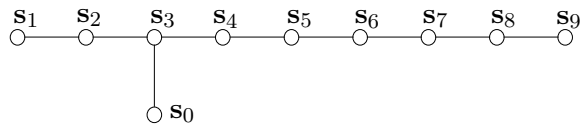
\caption{Dynkin diagram of $E_{10}$}
\label{F:coxetergraph}
\end{figure}

\medskip
Every element $\omega \in W(E_{10})$ can be written as a product of generators:
\[
\omega = g_1 g_2 \cdots g_k, \qquad g_i \in S.
\]
The \emph{length} $\ell(\omega)$ is the minimal such $k$, and the corresponding expression is called a \emph{reduced word} or \emph{reduced expression}. Elements of $W(E_{10})$ generally admit multiple reduced words.

A foundational result in Coxeter theory, due to Tits and Matsumoto, ensures that all reduced expressions for the same element are connected by a sequence of Coxeter (braid) relations.

\begin{thm}[Matsumoto's Theorem {\cite[Theorem~3.3.1]{BjornerBrenti}}]
Let $(W, S)$ be a Coxeter system and let $\omega \in W$.
\begin{enumerate}
\item Any expression $g_1 g_2 \cdots g_q$ with $g_i \in S$ representing $\omega$ can be transformed into a reduced expression by a sequence of \emph{nil-moves} (deleting adjacent squares $s_i^2 = 1$) and Coxeter (braid) relations.
\item Any two reduced expressions for $\omega$ are connected by a sequence of Coxeter relations.
\end{enumerate}
\end{thm}

\medskip
Let \(\mathbb{R}^{1,10}\) be a real vector space with orthonormal basis \((\vve_0, \vve_1, \dots, \vve_{10})\), where
\[
\vve_0^2 = 1, \qquad \vve_i^2 = -1 \;\; \text{for } i \ge 1, \qquad \vve_i \cdot \vve_j = 0 \;\; \text{for } i \ne j.
\]
This defines a symmetric bilinear form of signature \((1,10)\) on \(\mathbb{R}^{1,10}\).

Let \(V \subset \mathbb{R}^{1,10}\) be the subspace spanned by the ten simple roots \(\{ \alpha_0, \dots, \alpha_9 \}\), where
\[
\alpha_0 = \vve_0 - \vve_1 - \vve_2 - \vve_3, \qquad \alpha_i = \vve_i - \vve_{i+1} \quad \text{for } i = 1,\dots,9.
\]
This root system defines the \(E_{10}\) Dynkin diagram. The inner products among the roots satisfy
\[
B(\alpha_i, \alpha_j) := -\alpha_i \cdot \alpha_j = 
\begin{cases}
2 & \text{if } i = j, \\
-1 & \text{if } i \ne j \text{ and } i \sim j, \\
0 & \text{otherwise},
\end{cases}
\]
where \(i \sim j\) denotes adjacency in the diagram. Equivalently,
\[
B(\alpha_i, \alpha_j) = -2 \cos \left( \frac{\pi}{m_{ij}} \right).
\]

Each root \(\alpha_i\) defines a reflection \(\vs_i : V \to V\) given by
\[
\vs_i(v) := v - B(v, \alpha_i)\,\alpha_i = v + (v \cdot \alpha_i)\, \alpha_i.
\]
This is the standard reflection across the hyperplane orthogonal to \(\alpha_i\), satisfying \(\vs_i(\alpha_i) = -\alpha_i\) and fixing the orthogonal complement.

\medskip
\noindent
Define a linear representation
\[
\sigma : W(E_{10}) \to \mathrm{GL}(V), \qquad \sigma(s_i) := \vs_i.
\]
This representation is a faithful representation of \(W(E_{10})\) that preserve the bilinear form \(B\); see~\cite[Cor.~5.4]{Humphreys1990}. The action of $W(E_{10})$ on the Lorentzian space $V$ gives rise to the Tits cone, a fundamental geometric domain we will discuss via Hilbert metrics in Section~\ref{S:Hilbert}.

\medskip
\noindent\textbf{Spectral Radius.}
For each \(\omega \in W(E_{10})\), the linear operator \(\sigma(\omega)\) acts on \(V\) by a Lorentzian isometry with matrix representation \(M_\omega\) in the basis \(\{ \vve_0, \vve_1, \dots, \vve_{10} \}\). We define the \emph{spectral radius} of \(\omega\) by
\[
\rho(\omega) \;:=\; \rho(M_\omega),
\]
where \(\rho(M)\) denotes the spectral radius (largest modulus of eigenvalues) of the matrix \(M\).

\subsection{Bruhat Order}\label{SS:bruhat}

Let $(W,S)$ be a Coxeter system, and let
\[
R = \{\, \omega s \omega^{-1} \;:\; \omega \in W,\, s \in S \,\}
\]
denote the set of reflections in $W$.
For $\omega', \omega \in W$, we write
\[
\omega' \rightarrow \omega
\quad \text{if} \quad
\omega = \omega' r \quad \text{for some } r \in R,
\quad \text{and} \quad
\ell(\omega) > \ell(\omega').
\]

The \emph{(strong) Bruhat order} on $W$ is the partial order defined by:
\[
\omega' \le_{\text{Bruhat}} \omega
\quad \Longleftrightarrow \quad
\text{there exists a chain }
\omega'=\omega_0 \rightarrow \omega_1 \rightarrow \cdots \rightarrow \omega_k = \omega.
\]
Equivalently, $\omega' \le_{\text{Bruhat}} \omega$ if every reduced expression for $\omega$ contains a subword that is a reduced expression for $\omega'$.

\medskip
This partial order turns $(W, \le_{\text{Bruhat}})$ into a directed poset known as the \emph{Bruhat poset}, with unique minimal element the identity $1$.

\medskip
Bruhat order is graded by the length function and is compatible with multiplication:
\[
\omega' <_{\text{Bruhat}} \omega
\quad \Rightarrow \quad
\ell(\omega') < \ell(\omega),
\]
and for any $s \in S$,
\[
\omega' \le_{\text{Bruhat}} \omega
\ \text{ and }\
\ell(\omega') = \ell(\omega) - 1
\quad \Rightarrow \quad
\omega' s \le_{\text{Bruhat}} \omega s.
\]

We record a useful technical lemma that arises in spectral and length considerations:

\begin{lem}[{\cite[Lemma~2.2.10]{BjornerBrenti}}]
Suppose $x <_{\text{Bruhat}} x r$ and $y <_{\text{Bruhat}} r y$ for $x,y \in W$ and $r \in R$. Then
\[
xy <_{\text{Bruhat}} x r y.
\]
\end{lem}

\noindent\textbf{Unique Bruhat-minimal Representative} For $J \subset S$, the subgroup $W_J \subset W$ generated by the set $J$ is called a \emph{parabolic subgroup}. Define the set of minimal left coset representatives by
\[ W^J \;:=\; \{ \omega \in W\ :\ \ell(\omega s)>\ell(\omega) \ \text{for all } s \in J \}.\]
Then every $\omega \in W$ admits a unique factorization \[ \omega = uv, \qquad \text{with  }u \in W^J,\ v \in W_J,\]
satisfying the length additivity 
\[\ell(\omega) = \ell(u)+\ell(v).\] Moreover, each double coset $W_J \omega W_J$ contains a unique representative of minimal length with respect to the Bruhat order (see Deodhar~\cite{Deodhar1987}).

\section{$W(A_9)$-double cosets}\label{S:doubleCosets}

The finite Coxeter group \(W(A_9)\) embeds naturally as a (parabolic) subgroup of the
hyperbolic Weyl group \(W(E_{10})\) by restricting to an appropriate subset of simple reflections.
We regard
\[
W(A_9) \ \cong\ \langle s_1,s_2,\dots,s_9\rangle \ <\  W(E_{10}),
\]
as the stabilizer of the isotropic vector \(\vve_0\) in the standard basis
\(\{\vve_0,\vve_1,\dots,\vve_{10}\}\) of the Lorentzian lattice,
where
\[
\vve_0^2 = 1, \qquad \vve_i^2 = -1 \ (i\ge1), \qquad \vve_i\!\cdot \vve_j = 0 \ (i\ne j).
\]

\medskip
Let
\[
T \;=\; \bigl\{\, \{i,j,k\}\subset\{1,\dots,10\} \;\big|\; 1\le i<j<k\le 10 \,\bigr\}
\]
be the set of 3-subsets of \(\{1,\dots,10\}\).
For \(I=\{i,j,k\}\in T\), let \(\kappa_I\) denote the reflection in the root
\(\vve_0-\vve_i-\vve_j-\vve_k\).

Equivalently, if \(\sigma\in W(A_9)\cong S_{10}\) is any element with
\(\sigma(\{1,2,3\})=\{i,j,k\}\), then
\[
\kappa_{\{i,j,k\}} \;=\; \sigma\, s_0\, \sigma^{-1}.
\]
In particular, every \(\kappa_{\{i,j,k\}}\) is \(W(A_9)\)-conjugate to \(s_0=\kappa_{\{1,2,3\}}\).

\begin{lem}\label{L:kappaDecomposition}
For each element \(\omega \in W(E_{10}) \setminus W(A_9)\), 
there exist a positive integer \(n>0\) and triples 
\(I_1, I_2, \dots, I_n \in T\) such that
\[
\omega \;=\; \rho_L\, \kappa_{I_1}\, \kappa_{I_2}\, \cdots\, \kappa_{I_n}\, \rho_R,
\]
where \(\rho_L, \rho_R \in W(A_9)\).
\end{lem}

\begin{proof}
Since $W(E_{10})$ is generated by $W(A_9)$ together with $s_0$, 
every element $\omega \in W(E_{10}) \setminus W(A_9)$ can be written as
\[
\omega \;=\; \rho_{k+1} s_0 \rho_k s_0 \cdots s_0 \rho_1 s_0 \rho_0,
\qquad \rho_i \in W(A_9),
\]
where $k\ge 0$ denotes the number of occurrences of $s_0$ in the reduced word.

---

If $k=0$, then 
\[
\omega = \rho_1 s_0 \rho_0 = \rho_1\, \kappa_{\{1,2,3\}}\, \rho_0,
\]
which already has the desired form.

For $k=1$, we have
\[
\omega 
= \rho_2 \rho_1 \; (\rho_1^{-1} s_0 \rho_1)\; s_0 \rho_0.
\]
Let $I=\{i_1,i_2,i_3\}:=\rho_1^{-1}(\{1,2,3\})$.
Then $\rho_1^{-1} s_0 \rho_1$ is the reflection through 
\(\alpha_I = e_0 - e_{i_1}-e_{i_2}-e_{i_3}\),
hence $\rho_1^{-1} s_0 \rho_1 = \kappa_I$. 
Therefore
\[
\omega = (\rho_2\rho_1)\, \kappa_I\, s_0 \rho_0.
\]

---

The general case follows by induction on $k$.  
Define $\eta_j = \rho_j \rho_{j-1}\cdots\rho_1$ for $1\le j\le k+1$, 
and set $I_j = \eta_j^{-1}(\{1,2,3\})$. 
Then
\[
\eta_j^{-1} s_0 \eta_j = \kappa_{I_j},
\]
and substituting successively yields
\[
\omega 
= \eta_{k+1}\, (\eta_k^{-1}s_0\eta_k)\cdots(\eta_1^{-1}s_0\eta_1)\, s_0\rho_0
= \eta_{k+1}\, \kappa_{I_k}\cdots\kappa_{I_1}\, s_0\rho_0.
\]
Setting $\rho_L = \eta_{k+1}$ and $\rho_R = \rho_0$ 
gives the desired decomposition.
\end{proof}

\noindent
\textbf{Geometric interpretation.}
The decomposition in Lemma~\ref{L:kappaDecomposition} shows that 
each double coset 
\[
W(A_9)\,\omega\,W(A_9) \subset W(E_{10})\setminus W(A_9)
\]
admits a representative of the form 
\(\kappa_{I_1}\kappa_{I_2}\cdots\kappa_{I_n}\),
where each \(\kappa_{I_j}\) is a reflection through the hyperplane orthogonal to
\(\alpha_{I_j}=\vve_0-\vve_{i_1}-\vve_{i_2}-\vve_{i_3}\),
a $W(A_9)$-conjugate of the simple root \(\alpha_0\).
Geometrically, these double cosets parametrize the relative positions of 
$W(A_9)$-chambers in the Tits cone: composing successive $\kappa_{I_j}$ corresponds
to crossing hyperplanes parallel (under the $W(A_9)$-action) to the distinguished face
$H_0$ of the fundamental chamber.
Thus, each step in the product 
moves a base point deeper into the cone, increasing its 
$s_0$-level (defined in Section~3),
and the $W(A_9)$-double cosets provide a natural stratification of 
$W(E_{10})$ according to this relative height.

\medskip
\noindent
\textbf{$s_0$--length.}
For any element $\omega\in W(E_{10})$, define
\[
|\omega|_{s_0}
\ :=\
\min\Bigl\{\text{number of occurrences of }s_0
\text{ in a reduced expression of }\omega\Bigr\}.
\]
We call $|\omega|_{s_0}$ the \emph{$s_0$--length} of $\omega$.
Because braid relations (e.g.\ $s_0s_3s_0=s_3s_0s_3$) may change the raw
count of $s_0$ in a reduced word, we take the minimum over all
reduced expressions (Matsumoto's theorem ensures that all reduced words
are connected by braid moves).

Since each $\kappa_I$ ($I\in T$) is a $W(A_9)$-conjugate of $s_0$, every
$\kappa_I$ contains exactly one $s_0$ in any reduced expression.
Consequently, if $|\omega|_{s_0}=n$, then there exist triples
$I_1,\dots,I_n\in T$ such that
\[
\omega \ \in\ W(A_9)\,\kappa_{I_1}\kappa_{I_2}\cdots\kappa_{I_n}\,W(A_9),
\]
but
\[
\omega\ \notin\ W(A_9)\,\kappa_{I_1}\kappa_{I_2}\cdots\kappa_{I_j}\,W(A_9)
\qquad\text{for any } j<n.
\]

\medskip
\noindent
\subsection{Representatives of Double Cosets}\label{SS:Dcoset}
To investigate the algebraic structure of the double-coset space
\[
W(A_9)\backslash W(E_{10})/W(A_9),
\]
we begin by examining the relations among the reflections
\(\kappa_I\) for \(I\in T\).

\begin{lem}\label{L:kappa-relations}
Let \(I,J\in T\). Then the products of the corresponding reflections 
\(\kappa_I\) and \(\kappa_J\) satisfy the following relations:
\begin{enumerate}
\item If \(|I\cap J|=3\), then \(\kappa_I\kappa_J = \mathrm{id}\).
\item If \(|I\cap J|=2\), then 
\(\kappa_I\kappa_J =\kappa_I\,\rho= \rho\,\kappa_J,\)
where \(\rho\in W(A_9)\) is the reflection through the root 
\(\vve_i - \vve_j\), with \(i,j\) being the two indices not shared by \(I\) and \(J\).
\item If \(|I\cap J|=1\), then 
\(\kappa_I\kappa_J = \kappa_J\kappa_I\)
and \(|\kappa_I\kappa_J|_{s_0}=2\).
\item If \(|I\cap J|=0\), then 
\(\kappa_I\kappa_J\kappa_I = \kappa_J\kappa_I\kappa_J\)
and \(|\kappa_I\kappa_J|_{s_0}=2\).
\end{enumerate}
\end{lem}

These relations mirror the local Coxeter structure determined by the 
intersection pattern of the index sets \(I\) and \(J\).  
In particular, the cases \(|I\cap J|=2,1,0\) correspond respectively to 
types \(A_2\), \(A_1\times A_1\), and \(A_2\) within the 
subdiagram of the Coxeter graph generated by 
\(\kappa_I\) and \(\kappa_J\).

\begin{proof}
After relabelling indices if necessary, we may assume \(I=\{1,2,3\}\).
By considering 
\[
J=\{1,2,3\},\quad \{1,2,4\},\quad \{1,4,5\},\quad \text{and}\quad \{4,5,6\},
\]
one obtains representatives for the cases 
\(|I\cap J|=3,2,1,0\), respectively.
A direct computation using the defining action of \(\kappa_I\)
on the basis vectors \(\vve_0,\vve_1,\dots,\vve_{10}\)
confirms the stated relations in each case.
\end{proof}

\begin{cor}\label{C:small-s0}
Let $\omega \in W(E_{10})$. Then:
\begin{enumerate}
\item If $|\omega|_{s_0} = 1$, then $\omega \in W(A_9)\, \kappa_{\{1,2,3\}}\, W(A_9)$.
\item If $|\omega|_{s_0} = 2$, then either 
\[
\omega \in W(A_9)\, \kappa_{\{1,4,5\}} \kappa_{\{1,2,3\}}\, W(A_9)
\quad\text{or}\quad
\omega \in W(A_9)\, \kappa_{\{4,5,6\}} \kappa_{\{1,2,3\}}\, W(A_9).
\]
\end{enumerate}
\end{cor}

\begin{proof}
As noted earlier, for each $I \in T$, $\kappa_I$ is $W(A_9)$-conjugate to $s_0$.  
If $|\omega|_{s_0}=2$, then after relabelling the indices (via a permutation in $S_{10}$),
the conclusion follows immediately from Lemma~\ref{L:kappa-relations}.
\end{proof}

Another interesting relation arises when we consider three $\kappa$'s.

\begin{lem}\label{L:three-kappas}
Let $I,J,K \in T$ satisfy
\[
|I\cap J| = |J\cap K| = |K\cap I| = 1.
\]
Then the products of the corresponding reflections $\kappa_I$, $\kappa_J$, and $\kappa_K$
satisfy:
\begin{enumerate}
\item If \( |I\cap J\cap K| = 1 \), then \( |\kappa_I \kappa_J \kappa_K|_{s_0} = 3. \)
\item If \( |I\cap J\cap K| = 0 \), then \( |\kappa_I \kappa_J \kappa_K|_{s_0} = 2. \)
\end{enumerate}
\end{lem}

\begin{proof}
Suppose first that \( |I\cap J\cap K| = 1 \).
Without loss of generality, we may take
\[
I=\{1,6,7\},\qquad J=\{1,4,5\},\qquad K=\{1,2,3\}.
\]
A direct computation using the action of $\kappa_I$ on the basis vectors
\( \vve_0, \vve_1, \dots, \vve_{10}\) shows that
\[
\begin{aligned}
\kappa_I \kappa_J \kappa_K :\;
  & \vve_0 \mapsto 4 \vve_0 -3 \vve_1 - \vve_2 - \vve_3 - \vve_4 - \vve_5 - \vve_6 - \vve_7,\\
  & \vve_1 \mapsto 3 \vve_0 -2 \vve_1 - \vve_2 - \vve_3 - \vve_4 - \vve_5 - \vve_6 - \vve_7.
\end{aligned}
\]
Since any element of $W(A_9)$ acts as a permutation of the vectors
$e_1,\dots,e_{10}$, we see that
\[
\kappa_I \kappa_J \kappa_K
\notin
W(A_9)\, \{ \mathrm{id},\, \kappa_{\{1,2,3\}},\,
\kappa_{\{1,4,5\}} \kappa_{\{1,2,3\}},\,
\kappa_{\{4,5,6\}} \kappa_{\{1,2,3\}} \}\, W(A_9).
\]
By Corollary~\ref{C:small-s0}, it follows that
\( |\kappa_I \kappa_J \kappa_K|_{s_0} = 3. \)

\smallskip
If \( |I\cap J\cap K| = 0 \), we may take
\[
I=\{2,4,6\},\qquad J=\{1,4,5\},\qquad K=\{1,2,3\}.
\]
A direct computation shows that
\[
\kappa_I \kappa_J \kappa_K
 = \rho_1\, \rho_2\, \kappa_{\{4,5,6\}}\, \kappa_{\{1,2,3\}}\, \rho_3,
\]
where $\rho_1$, $\rho_2$, and $\rho_3$ are reflections through
$\vve_1-\vve_6$, $\vve_2-\vve_5$, and $\vve_3-\vve_4$, respectively.
Hence \( |\kappa_I \kappa_J \kappa_K|_{s_0} = 2. \)
\end{proof}

\medskip
\noindent
\textbf{Notation.}
To simplify notation, for $I_1,\dots,I_n \in T$ we write
\[
\kappa_{I_1,\dots,I_n}
\;:=\;
\kappa_{I_1}\,\kappa_{I_2}\cdots \kappa_{I_n},
\]
the product of the corresponding reflections in $W(E_{10})$.

\medskip
\begin{prop}\label{P:minimal-kappa}
For each integer $n \ge 0$, there exists a finite set $T_n$ of ordered $n$-tuples in $T$ with the following property:
for every element $\omega \in W(E_{10})$ satisfying $|\omega|_{s_0} = n$, there exists a unique
\[
(I_1,I_2,\dots,I_n) \in T_n
\]
such that
\[
\omega \in W(A_9)\,\kappa_{I_1,\dots,I_n}\,W(A_9).
\]
In particular, the double cosets
\[
\bigl\{\, W(A_9)\,\kappa_{I_1,\dots,I_n}\,W(A_9) \;:\; (I_1,\dots,I_n)\in T_n \bigr\}
\]
form a complete set of representatives for the elements of $W(E_{10})$ at $s_0$--length $n$.
\end{prop}

\begin{proof}
Since $T$ is finite, there are only finitely many ordered $n$-tuples $(I_1,\dots,I_n)\in T^n$,
and hence finitely many double cosets of the form
$W(A_9)\,\kappa_{I_1,\dots,I_n}\,W(A_9)$.
Choosing one representative $n$-tuple from each distinct double coset yields
a finite set $T_n$ satisfying the desired property.
\end{proof}

\begin{defn}[Transformation data]\label{D:trans-data}
For each ordered $n$-tuple $\widehat I = (I_1,\dots,I_n) \in T^n$,
consider the corresponding product of reflections
\[
\kappa_{\widehat I} := \kappa_{I_1}\kappa_{I_2}\cdots\kappa_{I_n}.
\]
The action of $\kappa_{\widehat I}$ on the basis
$\{e_0,e_1,\dots,e_{10}\}$
is represented by a matrix $M$.
We call this matrix representation the \emph{transformation data} of $\widehat I$,
and denote it by
\[
\mathrm{Trans}(\widehat I) := M.
\]
Since $\widehat I\in T^n$, the product $\kappa_{\widehat I}$ has
$s_0$--length $n$, that is,
\[
|\kappa_{\widehat I}|_{s_0} = n.
\]
We say that $\widehat I$ has \emph{$s_0$--length} $n$
if $|\kappa_{\widehat I}|_{s_0}=n$.
\end{defn}

\subsubsection{Representing Triples}\label{SSS:triples}

For $n=1$ and $n=2$, the sets $T_1$ and $T_2$ consist of one and two elements, respectively,
corresponding to the double cosets described in Corollary~\ref{C:small-s0}.
The relations among triples of reflections established in
Lemma~\ref{L:three-kappas}
govern the structure of $T_3$ and the higher-level sets $T_n$.
To construct $T_n$ explicitly and understand how successive triples interact,
we first analyze the adjacency relations among the triples in an ordered tuple.
This analysis provides the combinatorial data from which $T_n$ will later be defined recursively.

\medskip
According to Lemma~\ref{L:kappa-relations}, maintaining the $s_0$--length
imposes restrictions on the intersections of consecutive triples in $T_n$.

\begin{lem}[Adjacency constraint]\label{L:adj-gap1}
Let $\widehat I=(I_1,\dots,I_n)\in T^n$.
Then for each $k$,
\[
|I_k\cap I_{k+1}|\in\{0,1\}, \qquad 1\le k\le n-1.
\]
\end{lem}

\begin{proof}
Since $\widehat I=(I_1,\dots,I_n)\in T^n$, the corresponding product of reflections
$\kappa_{\widehat I}=\kappa_{I_1}\cdots\kappa_{I_n}$ has $s_0$--length equal to $n$.
The claim follows directly from Lemma~\ref{L:kappa-relations}.
\end{proof}

\subsubsection{Recursive construction of $T_n$ from $T_2$}\label{SSS:recursive}

We start from
\[
T_2 \;=\; \bigl\{\, I^{(1)}=(\{1,4,5\},\{1,2,3\}),\;
                 I^{(2)}=(\{4,5,6\},\{1,2,3\}) \,\bigr\},
\]
as provided by Corollary~\ref{C:small-s0}. For $T_3$, any ordered triple
$\widehat I=(I_1,I_2,I_3)$ of $s_0$--length $3$ must arise (up to the
$W(A_9)$--double coset relation) by adjoining a single triple
$I\in T$ to one of these pairs. Without loss of generality we place the
new triple first and write
\[
\widehat I \;=\; (\, I,\ I^{(j)}_1,\ I^{(j)}_2 \,), \qquad j\in\{1,2\}.
\]

\paragraph{Adjacency filter.}
By Lemma~\ref{L:adj-gap1}, minimality of the $s_0$--length forces
$|I\cap I^{(j)}_1|\in\{0,1\}$.
We first enumerate the finite set
\[
\mathcal{C}_j \;=\; \{\, I\in T \ :\ |I\cap I^{(j)}_1|\in\{0,1\}\,\}.
\]
(Practically, we retain those $I$ with $|I\cap I^{(j)}_1|=1$ as the primary
branch and treat the $0$--adjacency branch only when it survives the next test.)

\paragraph{Level test via the three--$\kappa$ lemma:}
For each $I\in\mathcal{C}_j$, apply Lemma~\ref{L:three-kappas} to the triple
$(I, I^{(j)}_1, I^{(j)}_2)$.
If all pairwise intersections are $1$, then
\[
|I\cap I^{(j)}_1|=|I^{(j)}_1\cap I^{(j)}_2|=|I^{(j)}_2\cap I|=1,
\]
and Lemma~\ref{L:three-kappas} distinguishes two outcomes:
\[
|I\cap I^{(j)}_1\cap I^{(j)}_2|=
\begin{cases}
1 &\Rightarrow\ |\kappa_I\kappa_{I^{(j)}_1}\kappa_{I^{(j)}_2}|_{s_0}=3 \quad\text{(keep)},\\[4pt]
0 &\Rightarrow\ |\kappa_I\kappa_{I^{(j)}_1}\kappa_{I^{(j)}_2}|_{s_0}=2 \quad\text{(discard)}.
\end{cases}
\]
If one of the three pairwise intersections is $0$, we compute
$|\kappa_I\kappa_{I^{(j)}_1}\kappa_{I^{(j)}_2}|_{s_0}$ directly (or via
short relations) and keep only those with level $3$.

\paragraph{Double--coset identification via matrices:}
For each surviving ordered triple
$\widehat I=(I,I^{(j)}_1,I^{(j)}_2)$, the transformation data $\mathrm{Trans}(\widehat I)$ is given by
the matrix representation $M$ of
$\kappa_I\kappa_{I^{(j)}_1}\kappa_{I^{(j)}_2}$ in the basis
$\{e_0,\dots,e_{10}\}$.
We identify two such matrices up to left/right multiplication by
$W(A_9)$, which acts by permuting the indices $1,\dots,10$ and fixes $e_0$:
\[
M_1 \approx M_2 \quad \Leftrightarrow \quad
M_1 = \rho\, M_2\, \sigma
\qquad \text{for some }\rho,\sigma\in W(A_9).
\]
Concretely, left multiplication by $\rho$ permutes rows corresponding to indices $1,\dots,10$,
and right multiplication by $\sigma$ permutes columns corresponding to indices $1,\dots,10$.
We choose one representative in each double coset
$W(A_9)\,\mathrm{Trans}(\widehat I)\,W(A_9)$ and record the corresponding $I$.
The resulting (finite) list yields $T_3$.

\medskip
\noindent
Repeating the same procedure inductively produces $T_{n+1}$ from $T_n$:

\medskip
\noindent\textbf{Inductive step $T_n\to T_{n+1}$.}
Given a representative $\widehat I'=(I_1,\dots,I_n)\in T_n$,
adjoin $I\in T$ in front to form $(I,\widehat I')$.
Filter by adjacency $|I\cap I_1|\in\{0,1\}$; apply the level test using the
relevant $\kappa$--relations (pair and triple lemmas) to ensure the $s_0$--length is $n+1$;
then identify double cosets via matrix representatives modulo left/right
$W(A_9)$--action as above. The surviving $I$'s determine $T_{n+1}$.
This procedure is straightforward to implement computationally.
A SageMath implementation is provided in Appendix~\ref{A:code}.

\subsection{Algebraic structure of the double coset}\label{SS:algebraic-structure}
The set of double cosets \(W(A_9)\backslash W(E_{10})/W(A_9)\) does not itself
carry a natural group structure.  
For \(x,y\in W(E_{10})\), the product
\[
(W(A_9)xW(A_9))\cdot (W(A_9)yW(A_9))
\]
is, in general, a finite \emph{disjoint union} of double cosets,
rather than a single one.
However, the free abelian group
\(\mathbb{Z}[\,W(A_9)\backslash W(E_{10})/W(A_9)\,]\)
admits a natural associative algebra structure under convolution:
\[
[W(A_9)xW(A_9)]\ast[W(A_9)yW(A_9)]
\;:=\;
\sum_{z\in{}^{A_9}\!W^{A_9}}
m_{x,y}^z\,[W(A_9)zW(A_9)],
\]
where each coefficient \(m_{x,y}^z\in\mathbb{Z}_{\ge0}\)
counts the number of ways to factor elements of the product
\((W(A_9)xW(A_9))\cdot(W(A_9)yW(A_9))\)
through the double coset \(W(A_9)zW(A_9)\).
This algebra is the \emph{parabolic Hecke algebra} associated with
\((W(E_{10}),W(A_9))\) at \(q=1\).

\medskip
Let $(W,S)$ be a Coxeter system and let $J \subset S$. 
Denote by $W_J$ the standard parabolic subgroup generated by~$J$.
It is well known \cite[Proposition~2.7(a)]{Billey:2018} that each double coset $W_J w W_J$ admits a unique
element of minimal length (equivalently, minimal with respect to the Bruhat order).
This follows from the general theory of parabolic
decompositions and the exchange condition
(see, e.g., \cite{Humphreys1990,Kazhdan-Lusztig,BjornerBrenti,Soergel1990,Deodhar1987,Dyer1991}).
In our setting, \(K\) denotes the set of all such minimal double--coset
representatives for \(W(A_9)\backslash W(E_{10})/W(A_9)\).

\medskip
For each integer \(n\ge 1\), set
\[
K_n \;=\; \{\, \kappa_{\widehat I} \;:\; \widehat I\in T_n \,\},
\qquad
K_0 \;=\; \{\mathrm{id}\}.
\]
Then the union
\[
K\ :=\ \bigcup_{n\ge 0} K_n
\]
forms a complete set of minimal--length double--coset representatives for
\[
W(A_9)\backslash W(E_{10})/W(A_9)
\]
with respect to the $s_0$--length filtration.  Explicitly,
\[
W(E_{10})
\;=\;
\bigsqcup_{n\ge 0}\;
\bigsqcup_{\kappa_{\widehat I}\in K_n}
W(A_9)\,\kappa_{\widehat I}\,W(A_9),
\]
and each element $\kappa_{\widehat I}$ has minimal possible $s_0$--length within its double coset.
Thus each Bruhat--minimal representative of
\(W(A_9)\backslash W(E_{10})/W(A_9)\) can be identified with a unique
element \(\kappa_{\widehat I}\in K\).
With this identification, we can endow \(K\) with a natural monoid structure
via the Demazure product.

\medskip
\noindent\textbf{Demazure product.}
Let \(K=\bigcup_{n\ge0} K_n\) denote the set of minimal double--coset
representatives for \(W(A_9)\backslash W(E_{10})/W(A_9)\).
For \(x,y\in K\), define the \emph{Demazure product}~$\diamond$ by
\[
x \diamond y
\;:=\;
\min_{\le_{\mathrm{Bruhat}}}
\bigl(K\cap W(A_9)\,x\,y\,W(A_9)\bigr),
\]
that is, the unique Bruhat--minimal element in the double coset containing~\(xy\).
Let
\[
\pi_D \colon W(E_{10}) \longrightarrow K
\]
be the projection sending each \(x\in W(E_{10})\) to the unique
minimal representative \(\kappa_{\widehat I}\in K\)
of the double coset containing~\(x\).

\medskip
The associativity of the Demazure product and the monoid homomorphism
property of $\pi_D$ hold for arbitrary Coxeter systems; see
Kazhdan--Lusztig~\cite{Kazhdan-Lusztig}, Dyer~\cite{Dyer1991}, and
Deodhar~\cite{Deodhar1987}. For finite Coxeter groups, these results
also appear in Geck--Pfeiffer~\cite[Chap.~8]{GeckPfeiffer}.

In particular, the set \(K\) of minimal double--coset
representatives in \(W(E_{10})\) is closed under the Demazure product~$\diamond$, and
\[
\pi_D(xy) = \pi_D(x) \diamond \pi_D(y),
\qquad x, y \in W(E_{10}).
\]

\begin{prop}\label{P:Demazure}
The Demazure product \(\diamond\) on \(K\) is associative and has identity element \(\mathrm{id} \in K\).
Moreover, the projection \(\pi_D : W(E_{10}) \to K\) is a monoid homomorphism:
\[
\pi_D(xy) = \pi_D(x) \diamond \pi_D(y) \qquad \text{for all } x, y \in W(E_{10}).
\]
\end{prop}

\begin{proof}
Each \(\kappa \in K\) is the unique Bruhat-minimal representative of its double coset \(W(A_9)\kappa W(A_9)\).
The compatibility of double coset multiplication with Bruhat order ensures that the minimal representative of a product of double cosets is determined by the Demazure product~$\diamond$.
These properties follow from the general theory of Hecke monoids and parabolic Bruhat decompositions; see, e.g., Dyer~\cite[Thm.~2.4]{Dyer1991}.
\end{proof}

\medskip
This product endows \(K\) with a natural monoid structure, encoding the
graded double--coset stratification of \(W(E_{10})\), and will serve as
the algebraic framework for the spectral and geometric analysis developed
in the following sections.

\section{The $s_0$-Level Filtration of $W(E_{10})$}\label{S:Ktree}
We begin by introducing a natural stratification of $W(E_{10})$ determined by the distinguished generator $s_0$.   We write
 \[ \conW\ :=\ \{ [\omega] \mid \omega \in W(E_{10})\, \}\] for the set of conjugacy classes in $W(E_{10})$, where $[\omega] \,=\,\{ x \omega x^{-1} | x \in W(E_{10})\}$. For a conjugacy class $[\omega] \in \conW\}$, we define its \emph{$s_0$-level}  by the function 
 \[ h_{s_0} : \conW \longrightarrow \mathbb{Z}_{\ge 0}, \qquad  h_{s_0} (\omega) \ :=\ \min_{v \in [\omega]} |v|_{s_0}.\] Thus $h_{s_0}(\omega)$ measures the smallest possible $s_0$-length among all representatives of the class. This invariant partitions the set of conjugacy classes of $W(E_{10})$ according to their minimal interaction with the distinguished reflection $s_0$. 
 
 \vspace{1ex}
We may now decompose the conjugacy set $\conW$ into strata according to
the $s_0$--level:
\[
\conW \;=\; \bigsqcup_{k\ge 0} F_k,
\qquad
F_k \;=\; \{\, [\omega] \in \conW \mid h_{s_0}(\omega)=k \,\},
\]
so that each $F_k$ consists of the conjugacy classes whose representatives
achieve minimal $s_0$--length~$k$.
We refer to this partition as the \emph{$s_0$--level decomposition} of
$W(E_{10})$.

\medskip
If $\omega'\in W(E_{10})$ is conjugate to $\omega$ with 
$|\omega|_{s_0}=k$, then
\[
|\omega'|_{s_0}\ge k.
\]
Hence
\[
\omega' \;\in\;
\bigsqcup_{n\ge k}\;
\bigsqcup_{\kappa_{\widehat I}\in K_n}
W(A_9)\,\kappa_{\widehat I}\,W(A_9).
\]
For each integer $n\ge0$, let $\mathcal{D}_n$ denote the union of
double cosets corresponding to the representatives
$\kappa_{\widehat I}\in K_n$:
\[
\mathcal{D}_n
\;:=\;
\bigsqcup_{\kappa_{\widehat I}\in K_n}
W(A_9)\,\kappa_{\widehat I}\,W(A_9).
\]
Thus $\mathcal{D}_n$ collects all elements of $W(E_{10})$
whose minimal $s_0$--length equals~$n$.
Using conjugacy in $W(E_{10})$, we can further decompose $\mathcal{D}_n$
by the $s_0$--level of the conjugacy class.

Since $F_k$ denotes the set of conjugacy classes in $W(E_{10})$
represented by elements of $s_0$--length $k$, we have
\begin{equation}\label{E:Dnn}
\mathcal{D}_n
\;=\;
\bigsqcup_{k\le n}\;
\mathcal{D}_{n,k},
\qquad
\mathcal{D}_{n,k}
\;:=\;
\{\,\omega\in\mathcal{D}_n : [\omega]\in F_k\,\}.
\end{equation}

\medskip
Geometrically, conjugation by elements of $W(A_9)$ corresponds to moving
the reflection hyperplane associated with $s_0$ within the $E_{10}$ Tits cone.
The $s_0$--level therefore measures the \emph{depth} of the orbit
in the hyperbolic direction, providing a natural stratification of
$W(E_{10})$ by increasing geometric complexity.

\medskip
Let $\rho(\cdot)$ denote the spectral radius of the linear action on
$\mathrm{span}\{e_0,\dots,e_{10}\}$.
Since our representation is faithful and conjugation acts by similarity,
$\rho$ is a class function on $W(E_{10})$.

For $n\ge0$ define the spectral set at $s_0$--level $n$ by
\[
\Lambda_n \;:=\; \{\, \rho(\omega)\ :\ [\omega]\in F_n \,\}.
\]
Equivalently, since $F_n$ consists of conjugacy classes whose minimal
representatives have $s_0$--length $n$, we have
\[
\Lambda_n \;=\; \{\ \rho(\omega)\ :\ \omega\in \mathcal{D}_{n,n}\ \}.
\]

More generally, for $k\le n$ set
\[
\mathrm{Spec}(\mathcal{D}_{n,k}) \;:=\; \{\, \rho(\omega)\ :\ \omega\in \mathcal{D}_{n,k}\,\}.
\]

\begin{lem}\label{L:spec-projection}
For every $n\ge k$ one has
\[
\mathrm{Spec}(\mathcal{D}_{n,k}) \ \subseteq\ \Lambda_k.
\]
In particular,
\[
\mathrm{Spec}(\mathcal{D}_n)
\;=\;
\bigcup_{k\le n}\mathrm{Spec}(\mathcal{D}_{n,k})
\ \subseteq\
\bigcup_{k\le n}\Lambda_k,
\qquad\text{and}\qquad
\Lambda_n \;=\; \mathrm{Spec}(\mathcal{D}_{n,n}).
\]
\end{lem}

\begin{proof}
If $\omega\in\mathcal{D}_{n,k}$, then $[\omega]\in F_k$ by definition.
Hence there exists $\omega'$ conjugate to $\omega$ with $|\omega'|_{s_0}=k$,
i.e.\ $\omega'\in\mathcal{D}_k$.
Since $\rho$ is conjugacy invariant, $\rho(\omega)=\rho(\omega')\in\Lambda_k$,
proving $\mathrm{Spec}(\mathcal{D}_{n,k})\subseteq \Lambda_k$.
The remaining assertions are immediate from
$\mathcal{D}_n=\bigsqcup_{k\le n}\mathcal{D}_{n,k}$ and the definition of $\Lambda_n$.
\end{proof}

\noindent
Consequently, \emph{new} spectral radii that first appear at level $n$ lie in
$\Lambda_n=\mathrm{Spec}(\mathcal{D}_{n,n})$; elements of $\mathcal{D}_{n,k}$ with $k<n$
can only contribute spectral radii already present at level $k$.

\subsection{Spectral refinement and growth graph of minimal representatives}\label{SS:spec-growth}

\medskip
The spectral set $\Lambda_n$ can be refined using the minimal
representatives in $T_n$.
For each $\widehat I\in T_n$, define
\[
\Lambda_n(\widehat I)
\;:=\;
\bigl\{\,
\rho(\omega)
\;:\;
\omega \in \mathcal{D}_{n,n}
\ \text{and}\
\omega \in W(A_9)\,\kappa_{\widehat I}\,W(A_9)
\,\bigr\}.
\]
Define an equivalence relation on $T_n$ by
\[
\widehat I \sim \widehat J
\quad\Longleftrightarrow\quad
\kappa_{\widehat I}=\kappa_{\widehat J}
\ \text{ or }\
\kappa_{\widehat I}=\kappa_{\widehat J}^{-1}.
\]
Since $\rho(\omega)=\rho(\omega^{-1})$ (the characteristic polynomial
is a Salem polynomial times cyclotomic factors), we have
$\Lambda_n(\widehat I)=\Lambda_n(\widehat J)$ whenever
$\widehat I\sim \widehat J$.
Thus
\[
\Lambda_n
\;=\;
\bigcup_{\,[\widehat I]\in T_n/\!\sim}\,
\Lambda_n(\widehat I),
\]
a (generally non-disjoint) cover of $\Lambda_n$ by source classes
modulo inversion.

\medskip
\noindent \textbf{$s_0$--growth graph and the spectral growth filtration.} Since the sets $T_n$ were constructed recursively
from $T_{n-1}$ by adjoining admissible triples
(cf.\ Section~\ref{SSS:recursive}),
the family $\{T_n/\!\sim\}_{n\ge0}$ forms a directed acyclic structure:
each vertex $[\widehat I]\in T_n/\!\sim$ is connected by directed edges
to the class of $(n\!+\!1)$--tuples in $T_{n+1}/\!\sim$ that extend it.
We refer to this directed system as the
\emph{$s_0$--growth graph of minimal representatives}.
Although different branches may merge
(two distinct extensions may lead to the same double coset),
the overall graph remains acyclic and records how new
minimal representatives---and hence new spectral contributions---arise
from lower levels.

\medskip
For each directed path
\[
[\widehat I^{(1)} ]\;\longrightarrow\;
[\widehat I^{(2)} ]\;\longrightarrow\;
\cdots \;\longrightarrow\;
[\widehat I^{(n)}],
\qquad
[\widehat I^{(m)}]\in T_m/\!\sim,
\]
in the growth graph, we obtain a sequence of
spectral subsets
\[
\Lambda_1(\widehat I^{(1)}) \;,\;
\Lambda_{2}(\widehat I^{(2)}) \;,\;
\cdots \;,\;
\Lambda_n(\widehat I^{(n)}),
\]
recording how the spectral radii evolve along the
recursive construction of minimal double--coset representatives.
We refer to this sequence as the
\emph{spectral growth filtration}.
It provides a combinatorial framework for tracing how
new spectral radii first appear in $\Lambda_n$ as $n$ increases,
and how they propagate through the overlapping branches
of the $s_0$--growth graph.

\begin{rem}
The spectral growth filtration traces the appearance of
new spectral values level by level, rather than forming
a nested inclusion of sets.
In particular, each $\Lambda_n(\widehat I)$ records
the spectral radii arising \emph{for the first time} at $s_0$--level~$n$.
The cumulative spectra \[ \Lambda_{\leq n}\ :=\ \bigcup_{k\le n} \Lambda_k\]
do form a nested family
$\Lambda_{\le1}\subseteq\Lambda_{\le2}\subseteq\cdots$,
reflecting the overall accumulation of spectral radii
across levels.
Thus, while the sequence
$\{\Lambda_k(\widehat I^{(k)})\}_{k\ge1}$ captures the
\emph{emergence dynamics} of individual spectral contributions,
the nested system $\{\Lambda_{\le n}\}_{n\ge1}$
describes their global spectral aggregation.
\end{rem}

\medskip
Using the recursive construction from
Section~\ref{SSS:recursive},
Figure~\ref{F:branchingGraph} illustrates the
$s_0$--growth graph for small $s_0$--levels as a branching graph.
Each node represents the newly adjoined triple along a directed path.
As shown in Figure~\ref{F:branchingGraph}, $s_0$--growth graph is a rooted directed acyclic graph.
For example, the leftmost downward path in
Figure~\ref{F:branchingGraph} corresponds to
\[
(1,2,3) \;\longrightarrow\; (1,4,5) \;\longrightarrow\; (1,6,7),
\]
which in terms of ordered sets of triples gives
\[
\{(1,2,3)\} \in T_1/\!\sim
\;\longrightarrow\;
\{(1,4,5),(1,2,3)\} \in T_2/\!\sim
\;\longrightarrow\;
\{(1,6,7),(1,4,5),(1,2,3)\} \in T_3/\!\sim.
\]
Two different nodes in the branching graph can merge.
For example,
\[
\{(5,6,7),(1,4,5),(1,2,3)\}
\;=\;
\{(1,5,7),(4,5,6),(1,2,3)\}
\in T_3/\!\sim.
\]
Hence there are merging branches emanating from the two nodes
$(1,4,5)$ and $(4,5,6)$ on the second level in
Figure~\ref{F:branchingGraph}.
To simplify the diagram, we label each vertex by the triple
appearing along the leftmost downward directed path.
In particular, the next node is recorded as $(5,6,7)$
(rather than $(1,5,7)$) to avoid duplication and confusion.

\medskip

In this way, the branching graph provides a combinatorial visualization
of the recursive generation of minimal double--coset representatives
and their associated spectral contributions.

\begin{figure}[h]
\centering
\begin{tikzpicture}[
  node distance=14mm and 10mm,
  >=Stealth,
  box/.style={
    rectangle,
    rounded corners,
    draw,
    font=\scriptsize,
    minimum width=13mm, 
    minimum height=4mm,
    align=center
  },
  ed/.style={->, shorten >=1.2mm, shorten <=1.2mm} 
]

\node[font=\small] at (-2.7,0) {$T_1/\!\sim$};
\node[font=\small] at (-2.7,-1.8) {$T_2/\!\sim$};
\node[font=\small] at (-2.7,-4.2) {$T_3/\!\sim$};

\node[box] (t1) at (0,0) {$(1,2,3)$};

\node[box] (a) at (0.9,-1.8) {$(1,4,5)$};
\node[box] (b) at (2.8,-1.8) {$(4,5,6)$};

\node[box] (c1) at (0.0,-4.2) {$(1,6,7)$};
\node[box] (c2) at (1.8,-4.2) {$(5,6,7)$};
\node[box] (c3) at (3.6,-4.2) {$(1,2,3)$};
\node[box] (c4) at (5.4,-4.2) {$(6,7,8)$};
\node[box] (c5) at (7.2,-4.2) {$(1,2,7)$};
\node[box] (c6) at (9.0,-4.2) {$(1,7,8)$};
\node[box] (c7) at (10.8,-4.2) {$(7,8,9)$};


\draw[ed] (t1) -- (a);
\draw[ed] (t1) -- (b);

\draw[ed] (a) -- (c1);
\draw[ed] (a) -- (c2);
\draw[ed] (a) -- (c4);

\draw[ed] (b) -- (c2);
\draw[ed] (b) -- (c3);
\draw[ed] (b) -- (c4);
\draw[ed] (b) -- (c5);
\draw[ed] (b) -- (c6);
\draw[ed] (b) -- (c7);

\end{tikzpicture}
\caption{
Branching graph for
$T_1/\!\sim \to T_2/\!\sim \to T_3/\!\sim$.
Each directed edge represents the extension process in which
a new ordered set of triples~$\widehat I$ is obtained by
adjoining the triple indicated at the end of the branch
as the first triple in the sequence.
For clarity, only the leftmost directed paths are shown,
so that each node is labeled by the triples appearing along that path.
}
\label{F:branchingGraph}
\end{figure}
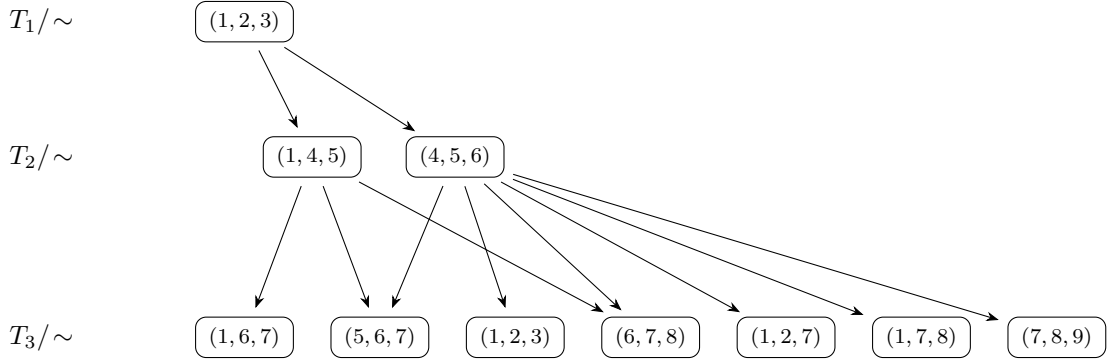

We summarize our discussion in Section~\ref{S:doubleCosets} and Section~\ref{S:Ktree} in the following theorem:

\begin{thm}[Triple-Labeled Graph Structure of Minimal Representatives]\label{T:TripleGraph}
Let $\mathcal{D}_{n,n}$ denote the set of elements in $W(E_{10})$ of $s_0$--level $n$ (as defined in~\eqref{E:Dnn}), and let $K \subset W(E_{10})$ be the set of Bruhat-minimal representatives of double cosets \(W(A_9)\backslash W(E_{10})/W(A_9)\).

Then:
\begin{enumerate}
\item Each $\kappa \in K \cap \mathcal{D}_{n,n}$ is uniquely determined by an ordere $n$--tuple of triples from
 \[T\ :=\ \{\{i,j,k\} \subset \{1,\dots, 10\} \ :\ 1\le i<j<k \le 10\}.\]
\item There exists a rooted directed acyclic graph (DAG) $\mathcal{G}$, whose depth-$n$ nodes correspond bijectively to a finite inductively defined subset $T_n \subset T^n$, such that each directed path of length $n$ encodes the recursive construction of a minimal double coset representative $\kappa \in \mathcal{D}_{n,n}.$
\end{enumerate}
\end{thm}

\section{Spectral Radii via the Hilbert Metric}\label{S:Hilbert}

The matrix representation of each $\omega \in W(E_{10})$ in the basis 
$\{\vve_0,\vve_1,\dots,\vve_{10}\}$ defines the geometric representation
\[
\sigma: W(E_{10}) \longrightarrow O(V),
\]
where $V\subset \mathbb{R}^{1,10}$ defined in Section~\ref{S:coxeter}.
The dual action on $V^*$ is given by
\[
\langle \omega(f),\, \omega(v)\rangle = \langle f, v\rangle,
\qquad 
\omega\in W(E_{10}),\; f\in V^*,\; v\in V,
\]
here $\langle f,v\rangle$ denotes the natural pairing of $f\in V^*$ with $v\in V$

Each simple reflection $s_i\in S$ acts as the reflection through the
hyperplane orthogonal to the simple root
\[
\alpha_0 = \vve_0 - \vve_1 - \vve_2 - \vve_3, 
\qquad
\alpha_i = \vve_i - \vve_{i+1} \quad (i=1,\dots,9).
\]
For each $\alpha_i$ let
\[
H_i := \{\,f\in V^* : \langle f,\alpha_i\rangle = 0\,\},
\]
with corresponding open half-spaces
\[
H_i^{\pm} := \{\,f\in V^* : \pm\langle f,\alpha_i\rangle > 0\,\}.
\]
The fundamental chamber is the closed convex cone
\[
F := \{\,f\in V^* : \langle f,\alpha_i\rangle \ge 0
\text{ for all } i=0,\dots,9\,\}.
\]
The \emph{Tits cone} of $W(E_{10})$ is the orbit
\[
U := \bigcup_{\omega\in W(E_{10})} \omega(F),
\]
on which $W(E_{10})$ acts properly by linear isometries.

\medskip

\paragraph{Hilbert metric.}
If $\mathcal{C}\subset V^*$ is a proper open convex cone and $x\neq y\in \mathcal{C}$, let $a,b\in\partial \mathcal{C}$ be the two intersection points of the line $\overline{xy}$ with $\partial \mathcal{C}$ ordered $a,x,y,b$; the Hilbert distance is
\[
d_\mathcal{C}(x,y)\;:=\;\log [a,x,y,b],
\]
where $[a,x,y,b]=\frac{|ax|\,|yb|}{|ay|\,|xb|}$ is the projective cross ratio.

\medskip
To follow McMullen's notation~\cite{McMullen:2002}, let $K$ denote the closure 
of the Tits cone $U$, and let $K^\circ$ denote its interior. 
By McMullen, the group $W(E_{10})$ acts discretely and isometrically 
for the Hilbert metric on the convex cone $K^\circ$, with 
$X = F \cap K^\circ$ as a fundamental domain.
Moreover, the spectral radius of $\omega$ is related to its 
Hilbert translation length as follows:

\begin{prop}[{\cite[Cor.~3.5]{McMullen:2002}}]\label{P:logSp}
For each $\omega\in W(E_{10})$,
\[
\log\rho(\omega)
= \inf_{x\in K^\circ} d_K(x,\omega x),
\]
where $d_K$ denotes the Hilbert metric on~$K$.
\end{prop}

One of the key advantages of the Hilbert metric on $K$ is that it satisfies the triangle inequality: for all $x,y,z \in K^\circ$,
\[ d_K(x,z) \le d_K(x,y) + d_K(y,z).\]
This immediately implies a submultiplicativity property for spectral radius. 

\begin{prop}\label{P:sub-mul}
For $\omega_1, \omega_2 \in W(E_10)$, 
\[ \rho(\omega_1 \omega_2) \;\le\; \rho(\omega_1) \rho(\omega_2).\]
\end{prop}

\begin{proof}
By the triangle inequality, \[ d_K(x, \omega_1\omega_2 x) \;\le\; d_K(x, \omega_2 x) + d_K(\omega_2 x, \omega_1( \omega_2 x)).\] 
Taking the infimum over all $x\in K^\circ$ on the left-hand side gives
\[ \inf_{x \in K^\circ} d_K(x, \omega_1\omega_2 x) \;\le \; d_K(x, \omega_2 x) + d_K(\omega_2 x, \omega_1 \omega_2 x).\] 
Again, taking the infimum  over all $x \in K^\circ$ to both terms separately on the right hand side, we have 
\[\inf_{x \in K^\circ} d_K(x, \omega_1\omega_2 x) \;\le \; \inf_{x \in K^\circ} d_K(x, \omega_2 x) + \inf_{y \in K^\circ} d_K(y, \omega_1 y).\]
Applying Proposition~\ref{P:logSp} to each term gives 
\[  \log \rho(\omega_1 \omega_2)\; \le\; \log \rho(\omega_1) + \log\rho(\omega_2),\] which is equivalent to the desired ineqaulity.
\end{proof}

Let $\omega \in W(E_{10})$ be given by a reduced word
\[
\omega = g_1 g_2 \cdots g_m, \qquad g_i \in S.
\]
We say that $\omega'$ is a \emph{subword} of $\omega$ if there exist 
indices $1 \le i_1 < i_2 < \cdots < i_k \le m$ such that
$\omega' = g_{i_1} g_{i_2} \cdots g_{i_k}$.

\medskip
Let $\pi: K^\circ \to X$ denote the natural projection sending each 
$x\in K^\circ$ to the unique point $\pi(x)\in X$ in its 
$W(E_{10})$-orbit. 
By McMullen~\cite[Section~4]{McMullen:2002}, one obtains the following 
useful conclusion, stated here in a form adapted to our setting.

\begin{prop}[{\cite[Section~4]{McMullen:2002}}]\label{P:mcmullen}
Suppose $\omega'$ is a subword of $\omega$. Then 
\[
\rho(\omega') \le \rho(\omega).
\]
\end{prop}

\subsection{Spectral Emergence Sequence}\label{SS:spectralG}

Let $\{\Lambda_k(\widehat I^{(k)})\}_{k\ge1}$ denote the sequence of
spectral subsets associated with a directed path in the
$s_0$--growth graph (cf.~Section~\ref{SS:spec-growth}).
This sequence records the emergence of new spectral values
as admissible triples are successively adjoined.

\begin{itemize}
\item[(1)] The initial node corresponds to
\[
\widehat I^{(1)} = \{(1,2,3)\} \in T_1.
\]

\item[(2)] For each $k\ge2$, the ordered set of triples
$\widehat I^{(k)}$ is obtained by adjoining a new triple
$J\in T$ to the beginning of $\widehat I^{(k-1)}=\{I_1,\dots,I_{k-1}\}$:
\[
\widehat I^{(k)} = \{J, I_1, \dots, I_{k-1}\}.
\]
Equivalently, $\widehat I^{(k)}$ extends $\widehat I^{(k-1)}$
by one admissible triple in the recursive construction
of $T_k$ from $T_{k-1}$ (see Section~\ref{SSS:recursive}).

\item[(3)] For each $k\ge1$, let
$\kappa_{\widehat I^{(k)}}$ denote the corresponding product of
reflections.  It is the unique minimal representative in its
double coset
\[
W(A_9)\,\kappa_{\widehat I^{(k)}}\,W(A_9).
\]

\item[(4)] The associated spectral subset
$\Lambda_k(\widehat I^{(k)})$ consists of the spectral radii
of all elements $\omega\in W(E_{10})$
that belong to the intersection
\[
\mathcal{D}_{k,k}\ \cap\
W(A_9)\,\kappa_{\widehat I^{(k)}}\,W(A_9),
\]
that is, those elements whose minimal $s_0$--length equals $k$
and which lie in the double coset of $\kappa_{\widehat I^{(k)}}$.
\end{itemize}

The sequence
$\{\Lambda_k(\widehat I^{(k)})\}_{k\ge1}$
thus describes the successive appearance of spectral radii
along a fixed directed path in the $s_0$--growth graph,
and forms the analytic counterpart of the combinatorial
branching structure developed in
Section~\ref{SS:spec-growth}.
\medskip
Define
\[
\mathcal{D}_k(\widehat I^{(k)})\ :=\
\mathcal{D}_{k,k}\ \cap\
W(A_9)\,\kappa_{\widehat I^{(k)}}\,W(A_9),
\]
so that
\[
\Lambda_k(\widehat I^{(k)}) \;=\;
\{\,
\rho(\omega)
\;:\;
\omega \in \mathcal{D}_k(\widehat I^{(k)})
\,\}.
\]
\begin{lem}\label{L:maxE}
For each $\widehat I^{(k)}\in T_k$, there exist elements
$\omega_k,\,\omega_k' \in \mathcal{D}_k(\widehat I^{(k)})$ such that
\[
\rho(\omega_k)
\;=\;
\max \Lambda_k(\widehat I^{(k)})
\qquad\text{and}\qquad
\rho(\omega_k')
\;=\;
\min \Lambda_k(\widehat I^{(k)}).
\]
Moreover, the element $\omega_k$ with maximal spectral radius is
conjugate to one of the form $\kappa_{\widehat I^{(k)}}\eta_k$,
where $\eta_k\in W(A_9)$:
\[
[\omega_k] = [\,\kappa_{\widehat I^{(k)}}\eta_k\,].
\]
\end{lem}

\begin{proof}
Since each double coset $W(A_9)\kappa_{\widehat I^{(k)}}W(A_9)$
is finite, $\Lambda_k(\widehat I^{(k)})$ attains both a maximum and
a minimum.  Let $\omega_k$ be the element achieving the maximal spectral
radius.  Because $\omega_k$ has $s_0$--length $k$ and
$\kappa_{\widehat I^{(k)}}$ is the unique minimal representative
of its double coset, there exists $\eta_k\in W(A_9)$ such that
$[\omega_k]=[\,\kappa_{\widehat I^{(k)}}\eta_k\,]$.
\end{proof}

\begin{lem}\label{L:subwords}
Let $\widehat I^{(k)}=\{J,I_1,\dots,I_{k-1}\}$.
Then for every $\omega\in\mathcal{D}_k(\widehat I^{(k)})$
there exists $\sigma\in W(A_9)$ such that
\[
[\omega] = [\,\sigma\,\kappa_J\,\omega_{k-1}\,],
\]
where $\omega_{k-1}$ is the element of
$\mathcal{D}_{k-1}(\widehat I^{(k-1)})$
with maximal spectral radius.
Moreover, $\sigma\,\kappa_J\,\omega_{k-1}$ is expressed in reduced word form.
\end{lem}

\begin{proof}
Let $\omega_{k-1}$ and $\eta_{k-1}\in W(A_9)$ be as in
Lemma~\ref{L:maxE}, so that
$[\omega_{k-1}]=[\,\kappa_{\widehat I^{(k-1)}}\eta_{k-1}\,]$.
If $\omega\in\mathcal{D}_k(\widehat I^{(k)})$, then
$\omega=\alpha\,\kappa_J\,\kappa_{\widehat I^{(k-1)}}\,\beta$
for some $\alpha,\beta\in W(A_9)$.
Conjugating by $\beta^{-1}\eta_{k-1}$ gives
\[
[\omega]
=
[\,\eta_{k-1}^{-1}\beta\alpha\,\kappa_J\,
  \kappa_{\widehat I^{(k-1)}}\,\eta_{k-1}\,].
\]
Since $\kappa_J\kappa_{\widehat I^{(k-1)}}$
is the unique minimal element in its double coset,
we may reduce the prefactor
$\eta_{k-1}^{-1}\beta\alpha$ to a single
$\sigma\in W(A_9)$, obtaining
\[
[\omega]=[\,\sigma\,\kappa_J\,\omega_{k-1}\,].
\]
\end{proof}

\begin{prop}\label{P:spectral-inequality}
Let $\{\Lambda_k(\widehat I^{(k)})\}_{k\ge1}$ be the sequence of
spectral subsets associated with a directed path in the
$s_0$--growth graph.
Then for every $\rho\in\Lambda_k(\widehat I^{(k)})$,
\[
\max \Lambda_{k-1}(\widehat I^{(k-1)})
\;\le\;
\rho
\;\le\;
\min \Lambda_{k+1}(\widehat I^{(k+1)}).
\]
\end{prop}

\begin{proof}
If $\rho\in\Lambda_k(\widehat I^{(k)})$, then there exists
$\omega\in\mathcal{D}_k(\widehat I^{(k)})$ with $\rho=\rho(\omega)$.
By Lemma~\ref{L:subwords}, a conjugate of $\omega_{k-1}$ appears as
a subword of $\omega$.  Since the spectral radius is invariant under
conjugation, Proposition~\ref{P:mcmullen} implies
\[
\rho \;\ge\; \rho(\omega_{k-1})
=\max \Lambda_{k-1}(\widehat I^{(k-1)}).
\]
Applying the same argument to $\Lambda_{k+1}(\widehat I^{(k+1)})$
yields the upper bound, completing the proof.
\end{proof}

\medskip
Proposition~\ref{P:spectral-inequality} reveals that the spectral
evolution along any directed path in the $s_0$--growth graph
is \emph{monotone and constrained}.
Each new level $k$ inherits its spectral range from the previous level,
bounded below by $\max\Lambda_{k-1}(\widehat I^{(k-1)})$
and above by $\min\Lambda_{k+1}(\widehat I^{(k+1)})$.
In geometric terms, as one adjoins successive reflections
$\kappa_J$ in the recursive construction,
the corresponding action on the Tits cone
expands the spectral radius only within a controlled interval.
This monotone interlacing of spectral sets
reflects the ordered structure of minimal double--coset representatives
and anticipates the continuous scaling behavior
governed by the Hilbert metric.

\begin{thm}[Spectral Monotonicity]\label{T:monotonicity}
Let $\{\Lambda_k(\widehat I^{(k)})\}_{k\ge1}$ be the sequence of
spectral subsets associated with a directed path in the
$s_0$--growth graph. Then the sequences of minima and maxima are non-decreasing in $s_0$--level $k$:
\[
\min \Lambda_k(\widehat I^{(k)})\;\le\;
\min \Lambda_{k+1}(\widehat I^{(k+1)}), \qquad
\max \Lambda_k(\widehat I^{(k)})\;\le\;
\max \Lambda_{k+1}(\widehat I^{(k+1)})
\quad \text{for all } k \ge 1.
\]
\end{thm}

\begin{proof}
This follows immediately from Proposition~\ref{P:spectral-inequality}, which ensures that
\[
\max \Lambda_k(\widehat I^{(k)}) \;\le\;
\min \Lambda_{k+1}(\widehat I^{(k+1)}).
\]
\end{proof}

\subsection{Levelwise Spectral Envelopes}\label{SS:spectralE}

For the initial node $\widehat I^{(1)}=\{(1,2,3)\}$,
\[
\Lambda_1([\widehat I^{(1)}])
\;=\;
\bigl\{\rho(\omega)\;:\;
\omega\in\mathcal D_{1,1}\cap W(A_9)\,\kappa_{\widehat I^{(1)}}\,W(A_9)\bigr\}.
\]
Since the double coset $W(A_9)\,\kappa_{\widehat I^{(1)}}\,W(A_9)$ is finite, 
the set $\Lambda_1([\widehat I^{(1)}])$ attains both a maximum and a minimum $>1$.\footnote{%
We exclude $\rho=1$, which arises from elliptic or parabolic elements through cyclotomic factors.}
Define
\[
\begin{aligned}
\Max_1([\widehat I^{(1)}]) &:= \max \Lambda_1([\widehat I^{(1)}]),\\
\Min_1([\widehat I^{(1)}]) &:= \min\{\rho\in\Lambda_1([\widehat I^{(1)}]) \;:\; \rho>1\}.
\end{aligned}
\]

By McMullen~\cite[Thm.~4.1]{McMullen:2002}, 
the minimum $\Min_1([\widehat I^{(1)}])$ equals 
\emph{Lehmer's number} $\lambda_L\approx1.17628$, 
the largest real root of the reciprocal polynomial
\[
L(t)\;=\;t^{10}+t^9-t^7-t^6-t^5-t^4-t^3+t+1.
\]
By direct computation, the maximum $\Max_1([\widehat I^{(1)}])=M_1$ is the largest real root of
\[
P(t)\;=\;t^{10}-t^9-t^8+t^7-t^5+t^3-t^2-t+1.
\]
In fact, there are exactly $11$ distinct spectral radii $>1$ in $\Lambda_1([\widehat I^{(1)}])$.
Table~\ref{T:Lambda_1} lists these values (rounded), together with their degrees and a compact encoding of the
\emph{reciprocal} minimal polynomials: for a reciprocal polynomial of even degree $2d$,
we list the first $d{+}1$ coefficients $(a_0,\dots,a_d)$, since the remaining $d$ are determined by palindromy.

\begin{table}[h]
\centering
\begin{tabular}{|c|c|c|l|}
\hline
$k$ & spectral radius $\rho_k$ & degree & reciprocal coeffs $(a_0,\dots,a_{d/2})$ \\ \hline\hline
1  & $1.17628$ & $10$ & $1,\,1,\,0,\,-1,\,-1,\,-1$ \\
2  & $1.21639$ & $10$ & $1,\,0,\,0,\,0,\,-1,\,-1$ \\
3  & $1.26123$ & $10$ & $1,\,0,\,-1,\,0,\,0,\,-1$ \\
4  & $1.28064$ & $\;\,8$ & $1,\,0,\,0,\,-1,\,-1$ \\
5  & $1.29349$ & $10$ & $1,\,0,\,-1,\,-1,\,0,\,1$ \\
6  & $1.35098$ & $10$ & $1,\,-1,\,0,\,0,\,-1,\,1$ \\
7  & $1.36000$ & $\;\,8$ & $1,\,-1,\,1,\,-2,\,1$ \\
8  & $1.38364$ & $10$ & $1,\,-1,\,0,\,-1,\,1,\,-1$ \\
9  & $1.40127$ & $\;\,6$ & $1,\,0,\,-1,\,-1$ \\
10 & $1.42501$ & $\;\,8$ & $1,\,-1,\,0,\,-1,\,1$ \\
11 & $1.43100$ & $10$ & $1,\,-1,\,-1,\,1,\,0,\,-1$ \\
\hline
\end{tabular}
\caption{Spectral radii in $\Lambda_1([\widehat I^{(1)}])$.  
Each row corresponds to the largest real root of a reciprocal polynomial determined by the listed first half of coefficients.}
\label{T:Lambda_1}
\end{table}

\medskip
For higher levels $n\ge2$, we aggregate the nodewise extrema over the equivalence classes $T_n/\!\sim$.
For each node class $[\widehat I]\in T_n/\!\sim$, define
\[
\begin{aligned}
\Max_n([\widehat I]) &:= \max \Lambda_n([\widehat I]),\\
\Min_n([\widehat I]) &:= \min \Lambda_n([\widehat I]).
\end{aligned}
\]

If $[\widehat J_1]\in T_{n-1}/\!\sim$ and $[\widehat J_2]\in T_{n+1}/\!\sim$
are connected to $[\widehat I]\in T_n/\!\sim$ by a directed path
\[
[\widehat J_1] \;\longrightarrow\; [\widehat I] \;\longrightarrow\; [\widehat J_2],
\]
then Proposition~\ref{P:spectral-inequality} yields
\[
\Max_{n-1}([\widehat J_1])
\;\le\;
\Min_n([\widehat I])
\;\le\;
\Max_n([\widehat I])
\;\le\;
\Min_{n+1}([\widehat J_2]).
\]

We now define the \emph{levelwise spectral envelopes}
\[
\begin{aligned}
M_n &:= \min_{[\widehat I]\in T_n/\!\sim}\ \Max_n([\widehat I]),
&\qquad&\text{(``bottleneck maximum'' at level $n$)},\\
m_n &:= \max_{[\widehat I]\in T_n/\!\sim}\ \Min_n([\widehat I]),
&\qquad&\text{(``bottleneck minimum'' at level $n$)}.
\end{aligned}
\]
By Proposition~\ref{P:spectral-inequality}, both sequences $(m_n)$ and $(M_n)$ are nondecreasing,
and every $\rho\in\Lambda_n$ lies in the global spectral band
\[
m_n \;\le\; \rho \;\le\; M_n.
\]
Since each level $T_n/\!\sim$ is finite, we can bound the levelwise envelopes
via the pathwise inequalities:
\[
M_{n-1}\ \le\ m_n\ \le\ M_n
\qquad(n\ge 2),
\]
where
\[
M_{n-1}=\min_{[\widehat J]\in T_{n-1}/\!\sim}\Max_{n-1}([\widehat J]),
\qquad
m_n=\max_{[\widehat I]\in T_n/\!\sim}\Min_n([\widehat I]).
\]

\medskip
In particular, $M_{n-1}$ provides a computable lower threshold for the
level-$n$ spectrum: every $\rho\in\Lambda_n$ satisfies $\rho\ge M_{n-1}$.
Thus, by enumerating $T_{n-1}/\!\sim$ and computing the nodewise maxima
$\Max_{n-1}$, we obtain an effective lower cutoff for~$\Lambda_n$.
Conversely, evaluating the nodewise maxima at level~$n$ yields
$M_n$, an upper envelope for~$\Lambda_n$.

\begin{thm}[Spectral Exhaustion Theorem]\label{T:exhaustion} For every $C >1$, there eixsts an integer $N\ge 1$ such that
\[ \{ \rho(\omega) \in \mathbb{R}\ :\ \omega \in W(E_{10}), \ \rho(\omega)\le C\} \subset \bigcup_{n\le N} \Lambda_n. \]
\end{thm}

\begin{proof}
By Theorem~\ref{T:monotonicity}, the level-wise spectral envelopes satisfy 
\[ m_1 \le m_2 \le \cdots \le m_n \le \cdots, \quad \text{with}\ \ \lim_{n \to \infty} m_n = \infty.\]
Hence, for any fixed $C>1$, there eixsts $N$ such that $m_N>C$, which implies that all spectral radii $\le C$ must appear in some $\Lambda_n$ for $n\le N$.
\end{proof}

\medskip
In terms of computation, it is nontrivial to determine whether a given
\(\omega \in W(A_9)\kappa_{\widehat I}W(A_9)\), with
\([\widehat I] \in T_n/\!\sim\), belongs to the \(s_0\)--level stratum \(F_n\).
To overcome this challenge, we compute a larger superset
\[
\tilde \Lambda_n
\ :=\
\bigl\{
\rho(\omega) \;:\;
\omega \in W(A_9)\,\kappa_{\widehat I}\,W(A_9),
\quad [\widehat I] \in T_n/\!\sim
\bigr\}.
\]
Since \(m_n \in [M_{n-1}, M_n]\), this superset \(\tilde \Lambda_n\) suffices to capture
all spectral radii in \(\Lambda_n\) up to the lower bound \(M_{n-1}\).

\medskip
To understand how new spectral values emerge across successive
\(s_0\)--levels, we define the extremal emerging radii at level~\(n\):
\[
\tilde m_n \;:=\;
\min\bigl( \tilde \Lambda_n \setminus \tilde \Lambda_{n-1} \bigr),
\qquad
M_n \;:=\; \max \tilde \Lambda_n.
\]
This perspective clarifies the spectral growth process.
As an example, we illustrate the levelwise spectral envelopes
for small~\(n\) in Table~\ref{T:Senvelops}.

\medskip
Notice that \(\tilde m_3\) is numerically slightly smaller than \(M_2\);
this occurs because \(\tilde m_3\) and \(M_2\) are drawn from
distinct directed paths in the \(s_0\)--growth graph. 
Let $\omega_\text{max} \in \mathcal{D}_{1,1}$ such that $\rho(\omega_\text{max}) = M_1$. Then by the monoid structure of $K$ given in Proposition~\ref{P:Demazure}, we have $\omega_\text{max}^k \in \mathcal{D}_{k,k}$. Since $\rho(\omega^k) = \rho(\omega)^k$, $k\ge 1$, from Proposition~\ref{P:sub-mul}, we have
\begin{lem} \label{L:maxM}
For each $k\ge 1$,
\[
M_k = M_1^k.
\]
\end{lem}
For other extreme values, our computations reveal the empirical inequalities
\[
\tilde m_k < M_{k-1}
\quad\text{and}\quad
\tilde m_k \approx \delta^{k-2} M_{k-1},
\qquad k \ge 3,
\]
where \(\delta := \tilde m_3/M_2 \approx 0.9335\).
Here ``\(\approx\)'' denotes numerical proximity rather than
an established asymptotic equivalence.
This scaling behavior suggests that the emergence of new spectral radii
is geometrically separated at higher levels.
Such exponential separation is reminiscent of displacement growth in Hilbert geometry; compare McMullen's metric perspective in~\cite{McMullen:2002}.

\begin{table}[h]
\centering
\begin{tabular}{|c|c|c|l|}
\hline
$\tilde m_i,\,M_i$ & spectral radius $\rho_k$ & degree & reciprocal coeffs $(a_0,\dots,a_{d/2})$ \\ \hline\hline
$\tilde m_1$ & $1.17628$ & $10$ & $1,\,1,\,0,\,-1,\,-1,\,-1$ \\
$M_1$ & $1.43100$ & $10$ & $1,\,-1,\,-1,\,1,\,0,\,-1$ \\
$\tilde m_2$ & $1.45799$ & $\;\,8$ & $1,\,0,\,-1,\,-1,\,0$ \\
$M_2$ & $2.04776$ & $10$ & $1,\,-3,\,3,\,-3,\,2,\,-1$ \\
$\tilde m_3$ & $1.91113$ & $10$ & $1,\,0,\,-2,\,-2,\,-1,\,-1$ \\
$M_3$ & $2.93035$ & $10$ & $1,\,-1,\,-4,\,-5,\,0,\,2$ \\
$\tilde m_4$ & $2.56366$ & $\;\,8$ & $1,\,-1,\,-2,\,-3,\,-4$ \\
$M_4$ & $4.19334$ & $10$ & $1,\,-3,\,-5,\,1,\,-2,\,-9$ \\
$\tilde m_5$ & $3.21255$ & $\;\,8$ & $1,\,-1,\,-4,\,-7,\,-7$ \\
$M_5$ & $6.00067$ & $10$ & $1,\,-6,\,-1,\,6,\,0,\,-1$ \\
\hline
\end{tabular}
\caption{Levelwise spectral envelopes.
Each entry corresponds to the largest real root of a reciprocal polynomial determined by the listed first half of coefficients.}
\label{T:Senvelops}
\end{table}

\medskip
These bounds enable complete enumeration of minimal polynomials for
\(\rho \in \Lambda_n\) below a given threshold.
To focus on new spectral contributions, we consider only
\emph{primitive} Salem numbers, i.e., those for which no positive root
\(\rho^{1/k}\) lies in \(\cup_n \Lambda_n\) for any \(k > 1\).

\medskip
Let \(\hat \Lambda_n\) denote the set of newly emerged primitive Salem numbers
at level~\(n\):
\[
\hat \Lambda_n
\ :=\
\{ \rho \in \tilde \Lambda_n \setminus \bigcup_{k < n} \tilde \Lambda_k\ :\ \rho \text{ is primitive.}\}.
\]
In Table~\ref{T:Lambda_1}, the eighth Salem number is not primitive. The first Salem number, $\rho_1$,is Lehmer's number---the largest real root of the polynomial \[ t^{10}+t^9-t^7-t^6-t^5-t^4-t^3+t+1.\] The minimal polynomial of $\rho_1^2$ is given by 
\[ t^{10}-t^9-t^7+t^6-t^5+t^4-t^3-t+1, \] which shows that $\rho_8=\rho_1^2$. 
Thus, the number of primitive Salem numbers at level one is \(|\hat \Lambda_1| = 10\).

Similarly, by checking the minimal polynomials of powers of Salem numbers, we find \(|\hat \Lambda_2| = 37\), \(|\hat \Lambda_3| = 180\),  \(|\hat \Lambda_4| = 866\) and  \(|\hat \Lambda_5| = 4100\). 

\vspace{1ex}
We include a complete list of all newly emerged primitive Salem numbers
in \(\tilde \Lambda_2\) in Appendix~\ref{B:tables}, along with the first 50 elements of \(\hat \Lambda_3\). The full lists of primitive Salem numbers in $\hat \Lambda_n, n=3,4,5$ are available in \texttt{https://www.math.fsu.edu/~kim/publication.html}.

\begin{rem}
The growth of \(\hat \Lambda_n\) appears to be exponential.
We observe:
\[
\frac{|\hat \Lambda_2|}{|\hat \Lambda_1|} = \frac{37}{10} = 3.7,
\quad
\frac{|\hat \Lambda_3|}{|\hat \Lambda_2|} = \frac{180}{37} \approx 4.86,
\quad
\frac{|\hat \Lambda_4|}{|\hat \Lambda_3|} = \frac{866}{180} \approx 4.81,
\quad
\frac{|\hat \Lambda_5|}{|\hat \Lambda_4|} = \frac{4100}{866} \approx 4.73.
\]
This suggests the heuristic:
\[
|\hat \Lambda_{n+1}| \approx \eta_n \cdot |\hat \Lambda_n|, \qquad \eta_n \approx 4-(n-1)/10
\qquad \text{for sufficiently large } n.
\]

\medskip
To gain further insight into the spectral density, we study the spacing
between consecutive newly emerging Salem numbers in
\[ \tilde \Lambda_n^\text{new} : = \tilde \Lambda_n \setminus \bigcup_{k<n} \tilde \Lambda_k.\]
Let \(\mathrm{Gap}_n\) denote the set of absolute differences between
successive elements of \(\tilde \Lambda_n^\text{new}\).
From computation, we observe:

\begin{table}[h]
\centering
\begin{tabular}{|c|c|c|}
\hline
$n$ & Average of $\mathrm{Gap}_n$& Standard Deviation of $\mathrm{Gap}_n$ \\ \hline\hline
$1$ & $0.0255$ & $0.0168$\\ \hline
$2$ & $0.0123$ & $0.0081$\\ \hline
$3$ & $0.0050$ & $0.0051$\\ \hline
$4$ & $0.0018$ & $0.0030$\\ \hline
$5$ & $0.0007$ & $0.0024$\\ \hline
\end{tabular}
\caption{Levelwise Gap distribution.}
\label{T:GapDis}
\end{table}
%
%
Moreover, the average gap appears to decay exponentially, as shown
in Figure~\ref{F:gap}. This suggests no arithmetic obstruction to
accumulation---another indication of the fractal-like density of
spectral radii within the Tits cone.

\end{rem}

\begin{figure}[h]
\centering
\includegraphics[width=0.7\textwidth]{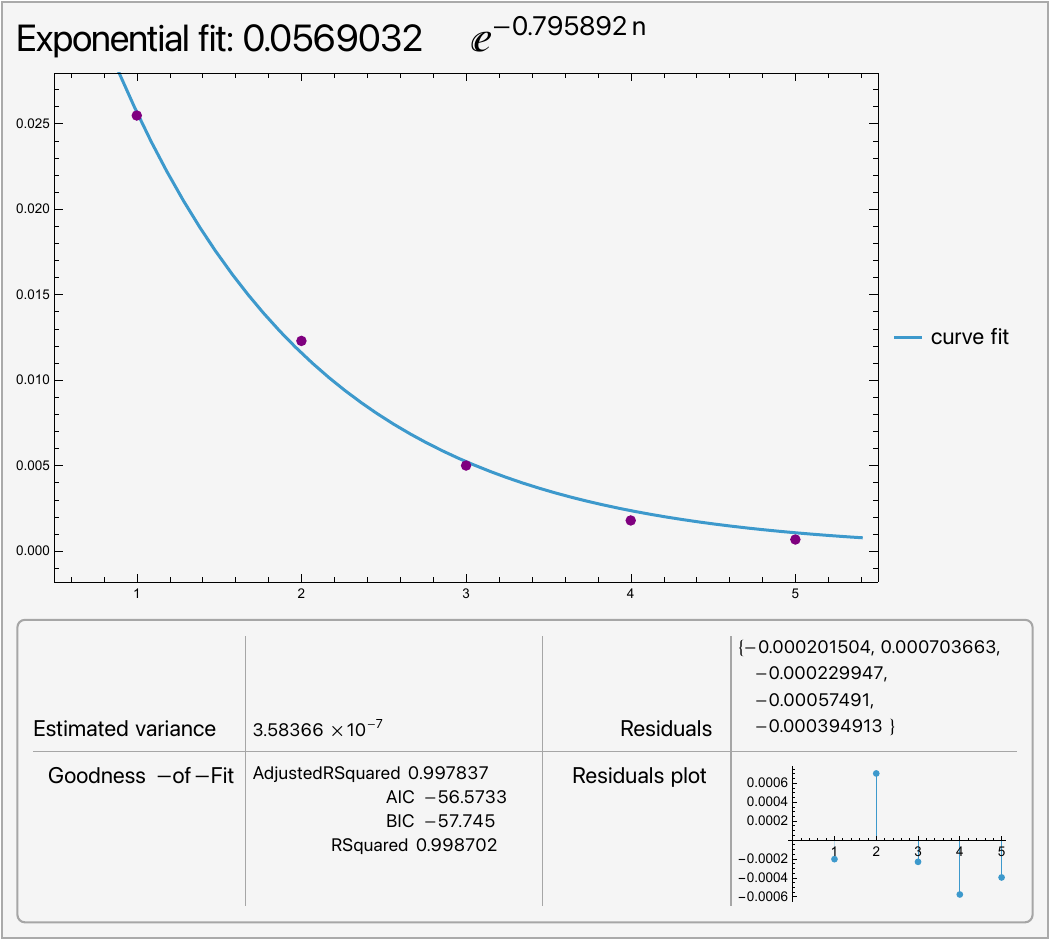}
\caption{Exponential decay fit to the average gaps in
\(\mathrm{Gap}_n\) between newly emerged Salem numbers in \(\tilde \Lambda_n^\text{new}\).
Purple dots show observed means; the curve is a best-fit exponential regression
obtained using Mathematica.}
\label{F:gap}
\end{figure}

\subsection{Hilbert--Metric Interpretation of Levelwise Growth}\label{SS:Hilbert}

The spectral envelopes $(m_n,M_n)$ introduced in 
Section~\ref{SS:spectralE} reflect the exponential growth rates of
elements in $W(E_{10})$ measured by their $s_0$--length.
We now interpret this growth geometrically in terms of
the Hilbert metric on the Tits cone.

\medskip
Let $\mathcal{C}\subset V$ be the Tits cone of $W(E_{10})$,
and let $\Omega\subset \mathbf{P}(V)$ denote its projectivization.
Choose an affine slice of $\Omega$ intersecting the interior of the
fundamental chamber in a bounded convex domain, so that the
\emph{Hilbert metric} $d_K$ is well defined on~$\Omega$.
Fix a basepoint $x$ in the interior of the fundamental chamber.

\medskip
For $\omega\in W(E_{10})$, define its \emph{Hilbert displacement} at a basepoint~$x$
by
\[
\delta_x(\omega)
\;:=\;
d_K\bigl(x,\,\omega\!\cdot\!x\bigr),
\]
where $d_K$ denotes the Hilbert metric on the interior $K^\circ$ of the Tits cone.
The \emph{Hilbert translation length} of~$\omega$ is
\[
\ell_K(\omega)
\;:=\;
\inf_{y\in K^\circ} d_K(y,\,\omega\!\cdot\!y).
\]
By McMullen~\cite[Cor.~3.5]{McMullen:2002}, one has $\ell_K(\omega)=\log\rho(\omega)$.

\medskip
Since each reflection in $W(E_{10})$ acts projectively on~$\Omega$ by
a cone--preserving linear transformation, 
the displacement $\delta(\omega)$ measures how far the element
$\omega$ moves the chamber in the hyperbolic direction determined by~$s_0$. Note that $\delta_x(\omega)$ depends on the basepoint, while $\ell_K(\omega)$ is a conjugation-invariant lower bound.

\medskip
For each level $n\ge1$, define the levelwise displacement bounds
\[
t_n
\;:=\;
\inf\{\delta(\omega)\;:\;\omega\in\mathcal{D}_{n,n}\},
\qquad
T_n
\;:=\;
\sup\{\delta(\omega)\;:\;\omega\in\mathcal{D}_{n,n}\}.
\]
Then the Hilbert balls
\[
B_n
\;:=\;
B_K(x,t_n)
\]
form a nested family
$B_1\subseteq B_2\subseteq\cdots$,
whose union fills $\Omega$.
Moreover, by construction,
the orbit points $\omega\!\cdot\!x$ with
$\omega\in\mathcal{D}_{n,n}$ all lie outside
$B_K(x,t_n-\varepsilon)$ for any $\varepsilon>0$;
that is, level~$n$ elements appear precisely beyond the
Hilbert radius~$t_n$.

\begin{thm}[Hilbert Ball Inclusion of Level-$n$ Elements]\label{T:hilbertBall}
Let $x$ be a point in the interior of the fundamental chamber.
Define $B_n := B_K(x, t_n) \subset \mathcal{C}$ to denote the Hilbert ball of radius.
\[
t_n := \inf\bigl\{\delta(\omega) \;:\; \omega \in \mathcal{D}_{n,n}\bigr\}.
\]
Then for any $\omega \in W(E_{10})$, if $\omega \cdot x \in B_n$,
the $s_0$--level of $\omega$ satisfies $h_{s_0}(\omega) \le n$.
\end{thm}

\medskip
This theorem highlights that the $s_0$--level filtration is not merely a combinatorial construct, but admits a geometric realization: the $s_0$--level of an element $\omega\in W(E_{10})$ provides a coarse bound on its hyperbolic distance from the base chamber in the Tits cone. In this way, the spectral growth hierarchy aligns with a nested system of Hilbert balls, offering a geometric model for the emergence of new spectral values and the expansion of reflection complexity at higher levels.

\medskip
By a theorem of McMullen~\cite[Prop.~4.1]{McMullen:2002}
(see also Bushell~\cite{Bushell:1973}),
for a cone--preserving linear transformation $A$ acting on~$\Omega$
the Hilbert displacement and spectral radius are logarithmically comparable:
there exist constants $c_1,c_2>0$, depending only on the cone,
such that
\[
c_1^{-1}\log\rho(A)
\;\le\;
d_K(x,A\!\cdot\!x)
\;\le\;
c_2\log\rho(A).
\]
Applying this to the elements of $\mathcal{D}_{n,n}$ yields
\[
c_1^{-1}\log m_n
\;\lesssim\;
t_n
\;\lesssim\;
T_n
\;\lesssim\;
c_2\log M_n.
\]
Hence, the Hilbert radii $(t_n,T_n)$ grow in lockstep with the
logarithmic spectral envelopes \[(\log m_n,\log M_n).\]
The $s_0$--level filtration of $W(E_{10})$ therefore induces a
\emph{geometric filtration} of the Tits cone by nested Hilbert balls,
in which each level $n$ corresponds to the maximal hyperbolic distance
realized by elements of $s_0$--length~$n$.

\medskip
This dual description reveals how the algebraic growth of
spectral radii in $\Lambda_n$ manifests as the hyperbolic
expansion of reflection walls $g\,s_0\,g^{-1}$ inside the Tits cone.
As $n\to\infty$, the balls $B_n$ exhaust~$\Omega$, providing a
geometric visualization of the spectral growth filtration
in the hyperbolic model of~$W(E_{10})$.

\begin{appendices}
\section{Spectral Radii in $W(E_{10})$}\label{B:tables}
This appendix lists primitive Salem numbers that appear as spectral radii
in the enlarged spectral sets \(\tilde\Lambda_2\) and \(\tilde\Lambda_3\),
introduced in Section~\ref{SS:spectralE}.
A Salem number \(\rho > 1\) is called \emph{primitive} if \(\rho^{1/k}\)
is not a Salem number for any integer \(k > 1\).
We exclude all non-primitive powers \(\rho^k\), \(k \ge 2\), from the tables.

Each entry in the tables corresponds to the largest real root of
a reciprocal minimal polynomial.
To conserve space, we list only the first half of the coefficients
\((a_0,\dots,a_d)\), as the remainder are determined by palindromy.
Entries are sorted in increasing order of spectral radius.
The table for \(s_0\)--level 2 is complete; the table for \(s_0\)--level 3
includes the first 50 primitive Salem numbers.

These spectral values represent the new growth of geometric complexity
within the \(W(E_{10})\) reflection group, as measured through the
spectral radius of elements in \(W(A_9)\kappa_{\widehat I}W(A_9)\).
We refer the reader to Section~\ref{SS:spectralE} and the spectral
growth filtration in Section~\ref{SS:spec-growth} for context and discussion.

{\footnotesize{
\begin{table}[h]
\centering
\begin{tabular}{|c|c|c|l||c|c|c|l|}
\hline
$k$ & spectral radius $\rho_k$ & degree & reciprocal coeffs  
& $k$ & spectral radius $\rho_k$ & degree & reciprocal coeffs  \\ \hline\hline
1 & $1.45799$ & $8$ & $  1,\,   0,\, -1,\, -1,\,   0$ &20 & $1.80979$ & $8$ & $  1,\, -1,\,   0,\, -2,\,   0$ \\
2 & $1.50614$ & $6$ & $  1,\, -1,\,   0,\, -1$ &21 & $1.81161$ & $8$ & $  1,\, -2,\,   0,\,   1,\, -1$ \\
3 & $1.52306$ & $8$ & $  1,\, -1,\, -1,\,   0,\,   1$ &22 & $1.83108$ & $6$ & $  1,\, -2,\,   0,\,   1$ \\
4 & $1.53293$ & $10$ & $  1,\, -1,\, -1,\,   0,\,   0,\,   1$ &23 & $1.83488$ & $8$ & $  1,\,   0,\, -1,\, -2,\, -3$ \\
5 & $1.54720$ & $8$ & $  1,\, -2,\,   2,\, -3,\,   3$ &24 & $1.86406$ & $8$ & $  1,\, -1,\, -2,\,   0,\,   2$ \\
6 & $1.55603$ & $6$ & $  1,\, -1,\, -1,\,   1$ &25 & $1.87574$ & $10$ & $  1,\, -2,\, -1,\,   3,\,   0,\, -3$ \\
7 & $1.58235$ & $6$ & $  1,\,   0,\, -1,\, -2$ &26 & $1.88320$ & $4$ & $  1,\, -2,\,   1$ \\
8 & $1.60545$ & $8$ & $  1,\, -2,\,   1,\,   0,\, -1$ &27 & $1.89911$ & $10$ & $  1,\, -2,\,   0,\,   0,\,   1,\, -1$ \\
9 & $1.63557$ & $6$ & $  1,\, -2,\,   2,\, -3$ &28 & $1.91650$ & $8$ & $  1,\, -1,\, -1,\, -1,\,   0$ \\
10 & $1.66105$ & $8$ & $  1,\, -2,\,   1,\, -1,\,   1$ &29 & $1.92063$ & $8$ & $  1,\, -3,\,   3,\, -2,\,   1$ \\
11 & $1.68491$ & $8$ & $  1,\, -1,\, -1,\,   0,\,   0$ &30 & $1.94686$ & $6$ & $  1,\, -1,\, -1,\, -1$ \\
12 & $1.69018$ & $10$ & $  1,\, -1,\, -2,\,   1,\,   1,\, -1$ &31 & $1.97209$ & $10$ & $  1,\, -2,\,   0,\,   0,\, -1,\,   3$ \\
13 & $1.69351$ & $8$ & $  1,\, -1,\,   0,\, -1,\, -1$  &32 & $1.97482$ & $6$ & $  1,\, -2,\,   1,\, -2$ \\
14 & $1.72208$ & $4$ & $  1,\, -1,\, -1$ &33 & $1.99400$ & $8$ & $  1,\, -2,\,   1,\, -2,\,   1$ \\
15 & $1.75310$ & $10$ & $  1,\,   0,\, -1,\, -1,\, -2,\, -3$ &34 & $1.99852$ & $10$ & $  1,\, -2,\,   1,\, -2,\,   1,\, -2$ \\
16 & $1.78164$ & $6$ & $  1,\, -1,\, -1,\,   0$ &35 & $2.01129$ & $8$ & $  1,\, -3,\,   3,\, -3,\,   3$ \\
17 & $1.79607$ & $8$ & $  1,\, -1,\, -1,\,   0,\, -1$ &36 & $2.01601$ & $10$ & $  1,\, -3,\,   2,\,   1,\, -3,\,   3$ \\
18 & $1.80017$ & $8$ & $  1,\, -3,\,   4,\, -5,\,   5$ &37 & $2.04249$ & $6$ & $  1,\, -3,\,   3,\, -3$ \\
19 & $1.80502$ & $10$ & $  1,\, -2,\,   0,\,   1,\, -1,\,   1$  & $$ & $$ & $  $ \\
\hline
\end{tabular}
\caption{Primitive Spectral radii in $\tilde \Lambda_2$.  
Each row corresponds to the largest real root of a reciprocal polynomial determined by the listed first half of coefficients.}
\label{T:Lambda_2}
\end{table}
}}

{\footnotesize{
\begin{table}[h]
\centering
\begin{tabular}{|c|c|c|l||c|c|c|l|}
\hline
$k$ & spectral radius $\rho_k$ & degree & reciprocal coeffs  
& $k$ & spectral radius $\rho_k$ & degree & reciprocal coeffs  \\ \hline\hline
1 & 1.91113 & 10 & 1,\,0,\,-2,\,-2,\,-1,\,-1 
& 26 & 2.17824 & 10 & 1,\,-3,\,3,\,-3,\,2,\,-1 \\
2 & 1.92679 & 8 & 1,\,0,\,-2,\,-2,\,-1 
& 27 & 2.18313 & 8 & 1,\,-1,\,-1,\,-2,\,-2 \\
3 & 1.94999 & 10 & 1,\,-1,\,-2,\,-1,\,1,\,3 
& 28 & 2.19565 & 6 & 1,\,-1,\,-1,\,-3 \\
4 & 1.95530 & 8 & 1,\,-2,\,-1,\,3,\,-1 
& 29 & 2.20113 & 6 & 1,\,-2,\,-1,\,2,\,-1 \\
5 & 1.98779 & 6 & 1,\,-2,\,-3,\,-2 
& 30 & 2.20346 & 10 & 1,\,-2,\,-1,\,1,\,0,\,1 \\
6 & 2.02203 & 8 & 1,\,-2,\,0,\,0,\,0 
& 31 & 2.20647 & 8 & 1,\,-2,\,-1,\,0,\,3 \\
7 & 2.04953 & 8 & 1,\,-1,\,-1,\,-1,\,-2 
& 32 & 2.21982 & 10 & 1,\,-3,\,2,\,0,\,-3,\,5 \\
8 & 2.05354 & 10 & 1,\,-2,\,-1,\,2,\,0,\,-1 
& 33 & 2.22587 & 6 & 1,\,-3,\,2,\,-1 \\
9 & 2.05632 & 10 & 1,\,-2,\,0,\,0,\,1,\,1 
& 34 & 2.22747 & 8 & 1,\,-1,\,-1,\,-2,\,-3 \\
10 & 2.06018 & 8 & 1,\,-3,\,2,\,1,\,-3 
& 35 & 2.22879 & 10 & 1,\,-1,\,-3,\,-1,\,2,\,3 \\
11 & 2.06973 & 8 & 1,\,-2,\,0,\,1,\,-3 
& 36 & 2.23749 & 8 & 1,\,-2,\,-1,\,1,\,0 \\
12 & 2.08102 & 4 & 1,\,-1,\,-2,\,-1 
& 37 & 2.23890 & 8 & 1,\,0,\,-2,\,-2,\,-4 \\
13 & 2.10429 & 8 & 1,\,-1,\,-1,\,-1,\,-3 
& 38 & 2.24275 & 8 & 1,\,-3,\,2,\,-1,\,1 \\
14 & 2.11854 & 8 & 1,\,0,\,-2,\,-2,\,-3 
& 39 & 2.24781 & 10 & 1,\,-1,\,-2,\,-1,\,-1,\,-1 \\
15 & 2.12150 & 8 & 1,\,-2,\,0,\,0,\,-1 
& 40 & 2.24885 & 8 & 1,\,-2,\,0,\,-1,\,0 \\
16 & 2.12479 & 10 & 1,\,-1,\,-3,\,0,\,2,\,1 
& 41 & 2.25119 & 8 & 1,\,-1,\,-2,\,-1,\,-1 \\
17 & 2.12992 & 10 & 1,\,-2,\,0,\,0,\,-1,\,0 
& 42 & 2.25646 & 6 & 1,\,-2,\,0,\,-1 \\
18 & 2.13466 & 8 & 1,\,-1,\,-1,\,-2,\,-1 
& 43 & 2.27019 & 10 & 1,\,-3,\,1,\,2,\,-3,\,0 \\
19 & 2.13709 & 8 & 1,\,-2,\,0,\,1,\,1 
& 44 & 2.27501 & 10 & 1,\,-2,\,0,\,0,\,-2,\,0 \\
20 & 2.14511 & 8 & 1,\,-1,\,-2,\,-1,\,1 
& 45 & 2.27671 & 8 & 1,\,-3,\,2,\,0,\,-2 \\
21 & 2.14850 & 10 & 1,\,-2,\,-2,\,4,\,1,\,-5 
& 46 & 2.27965 & 10 & 1,\,-2,\,0,\,0,\,2,\,0 \\
22 & 2.15372 & 4 & 1,\,-3,\,3,\,-3 
& 47 & 2.28496 & 10 & 1,\,-1,\,-1,\,-2,\,-3,\,-4 \\
23 & 2.16021 & 8 & 1,\,-1,\,-3,\,0,\,3 
& 48 & 2.29663 & 4 & 1,\,-2,\,0,\,-2 \\
24 & 2.16782 & 10 & 1,\,-1,\,-2,\,-2,\,1,\,2 
& 49 & 2.29949 & 10 & 1,\,-2,\,-2,\,3,\,1,\,-3 \\
25 & 2.17367 & 8 & 1,\,-2,\,0,\,1,\,-1 
& 50 & 2.30812 & 10 & 1,\,-2,\,0,\,-4,\,-5,\,-5 \\
\hline
\end{tabular}
\caption{First 50 primitive Salem spectral radii in $\tilde \Lambda_3$. Each row lists two entries side-by-side. The reciprocal coefficients column shows coefficients of $a_0, \dots a_{d/2}$ where $d=$ degree.}
\label{T:Primitive_Level3}
\end{table}
}
}

\vfill\eject
\section{SageMath Implementation}\label{A:code}

\noindent
This appendix provides a SageMath implementation of the recursive construction of the sets
$T_n$ described in Section~\ref{SSS:recursive}. The code defines the reflections
$\kappa_I$ for triples $I\subset\{1,\dots,10\}$, computes the transformation
data $\mathrm{Trans}(\widehat I)$, applies the adjacency and level filters, and
identifies matrices up to the double--coset action of $W(A_9)$ by canonicalizing
under independent row/column permutations (indices $1,\dots,10$; the $e_0$ row/column is fixed).

\medskip
\noindent\textbf{Usage.}
The initial set $T_2$ corresponds to the two double cosets in
Corollary~\ref{C:small-s0}.
Running \texttt{T3 = build\_T\_next(T2)} constructs $T_3$; iterating builds $T_{n+1}$ from $T_n$.
This procedure is straightforward to implement computationally. A SageMath implementation is provided below.

{\footnotesize{

\medskip
\noindent\textbf{SageMath code.}

\begin{lstlisting}
from sage.all import *
ZZ = IntegerRing()

J = diagonal_matrix([1] + [-1]*10)

def bilinear(x, y):
    x = vector(ZZ, x)
    y = vector(ZZ, y)
    return x.dot_product(J * y)

def reflection_matrix(v):
    v = vector(ZZ, v)
    den = bilinear(v, v)               
    assert den != 0
    outer = v.column() * (J * v).row() # 11x1 times 1x11 -> 11x11
    return identity_matrix(ZZ, 11) - (2/den) * outer

# --- Triples and reflections kappa_I ---
indices = list(range(1, 11))  # {1,...,10}
ALL_TRIPLES = [frozenset(S) for S in Subsets(indices, 3)]

def kappa_I_matrix(I):
    """
    I is a 3-subset of {1,...,10}. Reflection through v = e0 - sum_{i in I} ei.
    """
    v = [0]*11
    v[0] = 1
    for t in I:
        v[t] -= 1
    return reflection_matrix(v)

def Trans(tuple_of_triples):
    """
    Matrix of kappa_{I_1} ... kappa_{I_n} (left-multiplying, column convention).
    """
    M = identity_matrix(ZZ, 11)
    for I in tuple_of_triples:
        M = kappa_I_matrix(I) * M
    return M

# --- Adjacency filter: |I_$\cap$ I_{k+1}| in {0,1} ---
def adj_gap1_ok(tuple_of_triples):
    for A, B in zip(tuple_of_triples, tuple_of_triples[1:]):
        if len(A & B) not in (0, 1):
            return False
    return True

# --- Canonicalization under W(A9): independent row/col perms on 1..10 (0 fixed) ---
def canonicalize_double_coset(M):
    def sort_columns(N):
        # keys for cols 1..10: (entry at row 0, then rows 1..10)
        keys = [(N[0,j], tuple(N[1:11, j]), j) for j in range(1,11)]
        keys.sort()
        perm_old = [k[2] for k in keys]          # old indices in new order
        P = zero_matrix(ZZ, 11); P[0,0] = 1
        for new, old in enumerate(perm_old, start=1):
            P[old, new] = 1
        return N * P

    def sort_rows(N):
        # keys for rows 1..10: (entry at col 0, then cols 1..10)
        keys = [(N[i,0], tuple(N[i,1:11]), i) for i in range(1,11)]
        keys.sort()
        perm_old = [k[2] for k in keys]
        P = zero_matrix(ZZ, 11); P[0,0] = 1
        for new, old in enumerate(perm_old, start=1):
            P[new, old] = 1
        return P * N

    A = M
    for _ in range(6):        # typically stabilizes in 1--2 passes
        A1 = sort_columns(A)
        A2 = sort_rows(A1)
        if A2 == A:
            break
        A = A2
    return A

# --- Base set T2 from Corollary (two ordered pairs) ---
T2 = [
    (frozenset([1,4,5]), frozenset([1,2,3])),
    (frozenset([4,5,6]), frozenset([1,2,3])),
]

# --- Precompute canonical forms for level 2 double cosets (id, level-1, two level-2) ---
def precompute_level_le2_canon():
    forbidden = []

    # level 0: identity
    M_id = identity_matrix(ZZ, 11)
    forbidden.append(canonicalize_double_coset(M_id))

    # level 1: any kappa_I; all W(A9)-conjugate -> pick {1,2,3}
    M_1 = Trans((frozenset([1,2,3]),))
    forbidden.append(canonicalize_double_coset(M_1))

    # level 2: two classes from Corollary C:small-s0
    M_2a = Trans((frozenset([1,4,5]), frozenset([1,2,3])))
    M_2b = Trans((frozenset([4,5,6]), frozenset([1,2,3])))
    forbidden.append(canonicalize_double_coset(M_2a))
    forbidden.append(canonicalize_double_coset(M_2b))

    # immutable keys for fast comparison
    return { tuple(M.list()) for M in forbidden }

FORBIDDEN_LE2 = precompute_level_le2_canon()

def is_level3_by_exclusion(tup3):
    """
    tup3 = (I, I1, I2).
    Keep iff its canonicalized matrix is NOT in the precomputed level  2 set.
    """
    M = Trans(tup3)
    Mcanon = canonicalize_double_coset(M)
    key = tuple(Mcanon.list())
    return key not in FORBIDDEN_LE2

def admissible_extensions_front_T3(I1, I2):
    """
    Build all I with adjacency |I \cap I1| in {0,1}, then keep exactly those
    whose product with (I1, I2) has s0-length 3 by exclusion.
    """
    cand = [I for I in ALL_TRIPLES if len(I & I1) in (0,1)]
    out = []
    for I in cand:
        tup3 = (I, I1, I2)
        if not adj_gap1_ok(tup3):
            continue
        if is_level3_by_exclusion(tup3):
            out.append(I)
    return out

# --- Recursive builder: T_{n+1} from T_n ---
def build_T_next(Tn, *, strict_adjacency=True):
    """
    Input: Tn is a list of ordered tuples (each entry is a frozenset of size 3).
    Output: deduplicated list of ordered tuples of length n+1,
            modulo left/right W(A9) via canonicalization of Trans.
    """
    reps = {}
    for tup in Tn:
        n = len(tup)
        I1 = tup[0]
        # Special n=2 step: use the level 2 exclusion to decide true level 3
        if n == 2:
            cand = admissible_extensions_front_T3(I1, tup[1])
        else:
            cand = ALL_TRIPLES if not strict_adjacency else \
                   [I for I in ALL_TRIPLES if len(I & I1) in (0,1)]
        for I in cand:
            new_tup = (I,) + tup
            if strict_adjacency and not adj_gap1_ok(new_tup):
                continue
            M = Trans(new_tup)
            Mcanon = canonicalize_double_coset(M)
            key = tuple(Mcanon.list())
            if key not in reps:
                reps[key] = new_tup
    return list(reps.values())

# --- Example: build T3 from T2 ---
T3 = build_T_next(T2)
print("T3 size:", len(T3))
for t in T3:
    print(t)
\end{lstlisting}
}}

\end{appendices}

\bibliographystyle{plain}
\bibliography{biblio}
\end{document}